\documentclass{amsart}
\usepackage{bm,amsmath,amsthm,amssymb,mathtools,verbatim,amsfonts,tikz-cd,mathrsfs}
\usepackage{enumerate}
\usepackage{float}

\usepackage[hidelinks,colorlinks=true,linkcolor=blue, citecolor=black,linktocpage=true]{hyperref}
\usepackage{graphicx}
\usepackage[alphabetic]{amsrefs}
\usepackage{caption}
\usepackage{subcaption}
\usepackage{import}

\let\amsamp=&

\newtheorem{thm}{Theorem}[section]
\newtheorem{prop}[thm]{Proposition}

\newtheorem{cor}[thm]{Corollary}
\newtheorem{lem}[thm]{Lemma}

\theoremstyle{definition}
\newtheorem{define}[thm]{Definition}

\theoremstyle{remark}
\newtheorem{rem}[thm]{Remark}
\newtheorem{example}[thm]{Example}

\newtheorem{construction}[thm]{Construction}

\newcommand{\Z}{\mathbb{Z}}

\newcommand{\F}{\mathbb{F}}

\newcommand{\mc}[1]{\mathcal{{#1}}}
\newcommand{\p}{\partial}
\newcommand{\biota}{\bar{\iota}}
\newcommand{\im}{\mathrm{im}}
\newcommand{\id}{\mathrm{id}}
\newcommand{\C}{\mathcal{C}}

\newcommand{\A}{\mathcal{A}}
\newcommand{\sh}{\mathrm{sh}}

\newcommand{\rk}{\mathrm{rk}}
\newcommand{\gr}{\mathrm{gr}}

\newcommand{\R}{\mathcal{R}}

\newcommand{\I}{\mathfrak{I}}
\newcommand{\hI}{\widehat{\mathfrak{I}}}

\newcommand{\Mat}{\mathrm{Mat}}
\newcommand{\Cx}{\mathrm{Cx}}

\newcommand{\Ob}{\mathrm{Ob}}
\newcommand{\Mor}{\mathrm{Mor}}

\newcommand{\CI}{\mathcal{C} \mathcal{I}}
\newcommand{\tI}[1]{{#1} \tilde{\mathcal{I}}}
\newcommand{\tCI}{\mathcal{C} \tilde{\mathcal{I}}}

\graphicspath{{img/}}

\title{Almost $\iota$-complexes as immersed curves}
\author{Daniel Rostovtsev}
\address{Department of Mathematics\\California Institute of Technology\\  Pasadena, CA 91125, USA}
\email{drostovt@caltech.edu}

\begin{document}
	
\maketitle
		
\begin{abstract}
	Here the existence of a new homomorphism $P_\omega : \Theta_{\Z}^3 \to \Z$ is proven and the existence of a $\Z^{\infty}$ summand in $\Theta_{\Z}^3$ is reproven. This is done by approximating the involutive Heegaard Floer complexes of homology 3-spheres with immersed curves on the twice punctured disk.
\end{abstract}

\section{Introduction}

The classification of the integer homology cobordism group $\Theta_\Z^3$ is a major open problem in low dimensional topology. 
It was initially conjectured that the Rokhlin homomorphism $\mu: \Theta_\Z^3 \to \Z_2$ was an isomorphism before results in gauge theory identified an infinite subgroup \cite{fintushel_stern85} and then a $\Z^{\infty}$ subgroup \cite{furata, fintushel_stern90}. In 2002, it was shown with Yang-Mills theory that $\Theta_\Z^3$ has a $\Z$ summand \cite{froyshov}. 
Recently, Dai, Hom, Stoffregen and Truong used the involutive Heegaard Floer complexes defined in \cite{hendricks2016connected, Hendricks_2017} to prove the existence of a $\Z^\infty$ summand in $\Theta_\Z^3$ \cite{dai2018infiniterank}.

In \cite{hendricks2016connected}, Hendricks, Manolescu and Zemke define the group of $\iota$-complexes $\I$ to be all complexes $(C, \p, \iota)$ with a single $U$-tower in homology and $\iota^2 \simeq \id$ modulo local equivalence. The group operation of $\I$ is the tensor product of complexes. It is shown that the map $h: \Theta_{\Z}^3 \to \I$ taking $Y$ to $(CF^-(Y), \iota)$ is a homomorphism. Since it is unknown how to classify $\I$, Dai, Hom, Stoffregen, and Truong classified the group of almost $\iota$-complexes $\hI$, an approximation of $\I$ mod $U$, instead.

\begin{thm}[\cite{dai2018infiniterank}, Theorem 6.2] \label{thm:hi_classification}
	Every almost $\iota$-complex in $\hI$ is locally equivalent to a standard complex of the form $\C(a_1, b_2, \cdots, a_{2n-1}, b_{2n})$ where $a_i \in \{-,+\}$ and $b_i \in \Z \setminus \{ 0 \}$. 
\end{thm}

Dai, Hom, Stoffregen, and Truong defined a pivotal map $P: \hI \to Z$ which maps any local equivalence class to the grading of a terminal generator in the standard complex representative. They then proved that 
\begin{thm}[\cite{dai2018infiniterank}, Theorem 7.17] \label{thm:p_homo}
	$P$ is a homomorphism from $\hI$ to $\Z$.
\end{thm}
This allowed for the definition of maps $\phi_n: \hI \to \Z$ on each local equivalence class by $\phi_n(\C) = \# \{ b_i = n \} - \# \{ b_i = -n \}$. It was then shown using Theorem~\ref{thm:p_homo} that each $\phi_n$ is a homomorphism.
\begin{thm}[\cite{dai2018infiniterank}, Theorem 7.18] \label{thm:phi_n_homo}
	The maps $\phi_n : \hI \to \Z$ are homomorphisms.
\end{thm}
It was these $\phi_n$ that finally gave a splitting of $\Theta_\Z^3$ and witnessed an infinite rank summand. 
\begin{cor}[\cite{dai2018infiniterank}, Theorem 1.1] \label{cor:infinite_rank_summand}
	$\Theta_\Z^3$ has a $\Z^\infty$ summand.
\end{cor}

Streamlined proofs of Theorems \ref{thm:hi_classification}, \ref{thm:p_homo} and \ref{thm:phi_n_homo} are given here using the Fukaya category of the twice punctured disk. These proofs rely on the precurves described in \cite{Zibrowius_2020} and \cite{kotelskiy2019immersed}. From there, the proof of Corollary~\ref{cor:infinite_rank_summand} is the same as in \cite{dai2018infiniterank}. The precurve perspective also gives a new homomorphism $P_\omega : \hI \to \Z$ defined by $P_\omega(\C) = \#\{a_i = +\} - \#\{ a_i = - \}$. The existence of this homomorphism $P_\omega$ was conjectured by Dai, Hom, Stoffregen and Truong.

\begin{thm} \label{thm:p_omega_homo}
	The map $P_\omega : \hI \to \Z$ is a homomorphism.
\end{thm}

\begin{thm} \label{thm:p_omega_independent}
	$P_\omega : \hI \to \Z$ is surjective and not in the span of $\{ \phi_n \}$. 
\end{thm}

Once the classification is complete, it is shown that the homomorphisms defined in \cite{dai2018infiniterank} count the grading of a special generator in the precurve picture. Theorems~\ref{thm:p_omega_homo} and \ref{thm:p_omega_independent} are proven by endowing precurves from \cite{kotelskiy2019immersed} and \cite{Zibrowius_2020} with a natural bigrading.

The algebro-geometric correspondence in \cite{kotelskiy2019immersed} and \cite{Zibrowius_2020} is motivated by a similar construction of Haiden, Katzarkov, and Kontsevich in \cite{haiden2014flat}. The proof given in \cite{kotelskiy2019immersed} and \cite{Zibrowius_2020} that realizes chain complexes up to homotopy as decorated immersed curves draws heavily on the arrow sliding algorithm given by Hanselman, Rasmussen, and Watson in \cite{hanselman2016bordered}.

\subsection*{Organization}

The correspondence between almost $\iota$-complexes, precurves, and immersed curves is described in section~\ref{section:prelims}. In section~\ref{section:classifying_almost_iota_complexes}, the classification of $\hI$ is reproven by showing that a special immersed curve $\gamma_0(\tCI)$ contains all the information about the local equivalence class of $\C$. 
Finally, all known homomorphisms from $\hI$ to $\Z$ are described in section~\ref{section:homomorphisms} by counting the grading of a special generator in $\gamma_0(\tCI)$.

\subsection*{Acknowledgements} I would most of all like to thank Ian Zemke for taking the time to mentor me this summer, and for his help, encouragement, and advice. I would also like to thank Yi Ni for being a fantastic teacher, the Caltech Summer Faculty Program for sponsoring this research, Claudius Zibrowius for several helpful discussions, and Irving Dai, Jennifer Hom, Matthew Stoffregen and Linh Truong for the suggestion that $P_\omega$ might be a homomorphism. 

\section{Preliminaries} \label{section:prelims}

This section begins with a brief review of almost $\iota$-complexes and concludes with a crash-course in precurves. For a comprehensive description of almost $\iota$-complexes, see \cite{dai2018infiniterank}. For a comprehensive description of precurves, see \cite{kotelskiy2019immersed} and \cite{Zibrowius_2020}.

Here the convention will be that $\mc{R} = \F[U,Q]/(UQ)$. 

\subsection{Almost $\iota$-complexes} \label{subsection:almost_iota_complexes}

Motivated by the involution of the Heegaard Floer complex given in \cite{Hendricks_2017}, Hendricks, Manolescu and Zemke define the group $\I$ of $\iota$-complexes in \cite{hendricks2016connected}. More precisely,

\begin{define}[\cite{hendricks2016connected}, Definition 8.1] \label{def_iota_complex}
	An $\iota$-complex $\C = (C, \iota)$ is given by:
	\begin{enumerate}[$\bullet$]
		\item
		a free, finitely generated $\Z$-graded chain complex $\C$ over $\F[U]$ with 
		$$U^{-1} H_*(C) \cong \F[U, U^{-1}]$$
		where $U$ has degree $-2$ and $U^{-1} H_*(C)$ is supported in even gradings.
		\item
		a grading preserving, $U$-equivariant chain homomorphism $\iota : C \to C$ such that $\iota^2 \simeq \id$. 
	\end{enumerate}
\end{define}

The group $\I$ is the set of $\iota$-complexes modulo local equivalence, which is defined below, with the group operation being tensor product. 

\begin{define}[\cite{hendricks2016connected}, Definition 8.5] \label{def_local_equivalence}
	Two $\iota$-complexes $\C, \C'$ are \textit{locally equivalent} if there exist grading preserving chain maps $f: C \to C', g: C' \to C$ such that 
	$$f \circ \iota \simeq \iota' \circ f, \qquad g \circ \iota' \simeq \iota \circ g$$
	and $f, g$ induce isomorphisms between $U^{-1} H_*(C)$ and $U^{-1} H_*(C')$. 
\end{define}

It can be shown that $\otimes : \I \times \I \to \I$ forms a group operation that respects local equivalence.  In essence, $\I$ contains all possible complexes that could arise as the involutive Heegaard Floer complex $CFI^-$ of some 3-manifold $M$. The classification of $\I$ is not yet complete, although it would be a significant step towards the classification of the homology cobordism group. It is known that $\I$ has torsion free elements and elements which are 2-torsion, but it is not known if $\I$ has $n$-torsion for any $n>2$. 

To further understand $\I$, the authors of \cite{dai2018infiniterank} define a group $\hI$ of almost $\iota$-complexes, which is an approximation of $\I$ modulo $U$ up to local equivalence. To define $\hI$ it is necessary to introduce the concept of a homotopy mod $U$. Compare the following to \cite{dai2018infiniterank}:

\begin{define}
	Two grading preserving $\F[U]$-module homomorphisms $f,g: C \to C'$ are homotopic mod $U$ if there exists an $\F[U]$ module homomorphism $H: C \to C'$ such that $H$ increases grading by one and 
	$$f + g + [H, \p] \in \im U$$
where $[H, \p]$ denotes the commutator of $H$ and $\p$.
\end{define}

$\hI$ is the set of almost $\iota$-complexes (defined below) modulo a version local equivalence which uses homotopy mod $U$.

\begin{define}[\cite{dai2018infiniterank}, Definition 3.2] \label{def_almost_iota_complex}
	An \textit{almost $\iota$-complex} $\C = (C, \biota)$ is given by:
	\begin{enumerate}[$\bullet$]
		\item
		a free, finitely generated $\Z$-graded chain complex $\C$ over $\F[U]$ with 
		$$U^{-1} H_*(C) \cong \F[U, U^{-1}]$$
		where $U$ has degree $-2$ and $U^{-1} H_*(C)$ is supported in even gradings.
		\item
		a grading preserving $\F[U]$ module homomorphism $\biota: C \to C$ such that
		$$ \biota \circ \p + \p \circ \biota \in \im U \qquad \text{ and } \qquad \biota^2 \simeq \id \mod U $$
	\end{enumerate}
\end{define}

\begin{define}[\cite{dai2018infiniterank}, Definition 3.5]
	Two almost $\iota$-complexes $\C_1 = (C_1, \biota_1)$ and $\C_2 = (C_2, \biota_2)$ are \textit{locally equivalent} if there are grading-preserving, $U$-equivariant chain maps $f: C \to C', g: C' \to C$ such that
	$$f  \circ \biota \simeq \biota' f \mod U, \qquad g \circ \biota' \simeq \biota \circ g \mod U$$
	and$f, g$ induce isomorphisms between $U^{-1}H_*(C)$ and $U^{-1}H_*(C')$. 
\end{define}

It can be shown that $\hI$ is a group with respect to tensor product by using the same argument as in \cite{hendricks2016connected} for $\I$. The morphisms in $\hI$ are defined as follows:

\begin{define} [\cite{dai2018infiniterank}, Definition 3.4]
	An \textit{almost $\iota$-morphism} is a grading preserving $U$-equivariant chain map $f : \C \to C'$ between almost $\iota$-complexes such that $f \biota \simeq \biota' f \mod U$.
\end{define}

 Local equivalence classes in $\hI$ have special representatives that Dai, Hom, Stoffregen, and Truong call standard standard complexes. These are defined below. Note the convention that $\omega = 1 + \biota$.

\begin{define} [\cite{dai2018infiniterank}, Definition 4.1]
	A \textit{standard complex} 
	$$\C(a_1, b_2, \cdots, a_{2n-1}, b_{2n})$$ 
	is an almost $\iota$-complex parametrized by $a_i \in \{-, +\}$ and $b_i \in \Z - \{ 0 \}$. It is generated by $T_0, \cdots, T_{2n}$. The $b_i$ specify the differential. Let $\p T_{2i} = U^{b_{2i}} T_{2i-1}$ if $b_{2i}>0$ and let $\p T_{2i-1} = U^{b_{2i}} T_{2i}$ if $b_{2i}<0$. The $a_i$ specify the involution. Let $\omega T_{i} = T_{i-1}$ if $a_i = +$ and let $\omega T_{i-1} = T_{i}$ if $a_i = -$. 
\end{define}

Recall that Dai, Hom, Stoffregen and Truong proved that every almost $\iota$-complex is locally equivalent to a standard complex. This is the content of Theorem~\ref{thm:hi_classification}. 

Standard complexes are convenient to visualize as chains alternating between $\p$ and $\omega$ arrows, which always start with an $\omega$ arrow and end with a $\p$ arrow. An example is given in figure~\ref{fig:local_equiv_counter_example_a} for the standard complex $\C(+,-2)$.

The standard complex representative of an almost $\iota$-complex is difficult to find in general. A significant portion of \cite{dai2018infiniterank} is dedicated to overcoming this obstacle. This motivates the following definition: 

\begin{define}[\cite{dai2018infiniterank}, Definition 3.22] \label{def:reduced}
	An almost $\iota$-complex $\C$ is \textit{reduced} if $\p \equiv 0 \mod U$. 
\end{define}
A key lesson in \cite{dai2018infiniterank} is that reduced almost $\iota$-complexes are easier to work with than standard complexes. 
It is not difficult to show that an almost $\iota$-complex can be assumed to be reduced without loss of generality. This relies on the classification of chain complexes over a PID, and mirrors the classification of modules over a PID. The interested reader may consult \cite{hendricks2016connected} or \cite{dai2018infiniterank}. It is proven in \cite{dai2018infiniterank} that
\begin{lem}[\cite{dai2018infiniterank}, Theorem 6.2] \label{lem:almost_iota_locally_equivalent_to_reduced}
	Every almost $\iota$-complex is locally equivalent to a reduced almost $\iota$-complex. 
\end{lem}
The following properties of reduced almost $\iota$-complexes will be useful in realizing almost $\iota$-complexes as immersed curves.
\begin{lem}[\cite{dai2018infiniterank}, Lemma 3.23] \label{lem:reduced_has_omega_squard_equiv_0_mod_U}
	Every reduced almost $\iota$-complex has that $\omega^2 \equiv 0 \mod U$. 
\end{lem}
\begin{lem}[\cite{dai2018infiniterank}, Lemma 3.23] \label{lem:reduced_has_w_f_commute_mod_U}
	If $f : \C \to \C'$ is a chain map between two reduced almost $\iota$-complexes, then $f$ is an almost $\iota$-morphism if and only if $[\omega, f] \equiv 0 \mod U$. 
\end{lem}
Section~\ref{subsection_almost_iota_complexes_as_immersed_curves} will describe how Lemma~\ref{lem:reduced_has_omega_squard_equiv_0_mod_U} can be used to realize almost $\iota$-complexes as precurves in full detail. In sections~\ref{subsection_primitive_representatives} and \ref{subsection_classification}, the local equivalence class of $\C$ will be extracted from its corresponding precurve. For that, the following splitting lemmas will be needed:

\begin{define}
	An almost $\iota$-complex $\C = (C, \biota)$ splits if $C$ can be written as $C' \oplus B$ with $\biota C' \subset C', \biota B \subset B, \p C' \subset C',$ and $\p B \subset B$, where without loss of generality, the generator of the infinite tower in $H_*(C)$ lies in $C'$. Then it is said that $\C = \C' \oplus \mc{B}$, where $\C' = (C', \p|_{C'}, \biota|_{C'})$ and $\mc{B} = (B, \p|_{B}, \biota|_{B})$. 
\end{define}

\begin{lem}
	If an $\iota$-complex $\C$ splits into $\C' \oplus \mc{B}$, then $\C'$ is an almost $\iota$-complex.
\end{lem}
\begin{proof}
	It is evident that $\C'$ is an $\F[U]$ chain complex with a single $U$-tower in homology, with $[\biota|_{C'}, \p|_{C'}] \in \im U$. Now it must only be shown that $\biota|_{C'}^2 \simeq \id \mod U$. Let $H$ be the $\F[U]$-module homomorphism witnessing $\biota^2 \simeq \id \mod U$, and $i : C' \to C$ and $\pi: C \to C'$ be the natural inclusion and projection maps Then 
	$$ \id_C + \biota^2 + [H, \p] \in \im U.$$
	Pre-composing with $i$ and post-composing with $\pi$ gives
	$$ \pi \id_C i + \pi \biota^2 i + \pi [H, \p] i \in \im U.$$
	It is true that $\pi \id_C i = \id_{C'}$, $\pi \biota^2 i = (\pi \biota i)^2$ and $\pi \biota i = \biota|_{C'}$, and
	$$ \pi [H, \p] i = \pi H \p i + \pi \p H i = \pi H i \p + \p \pi H i = [\pi H i, \p|_{C'}]$$
	as $\pi$ and $i$ commute with $\p$. Thus
	$$ \id_{C'} + (\biota|_{C'})^2 + [\pi H i, \p|_{C'}] \in \im U$$
	and $\C' = (C' ,\p|_{C'}, \biota|_{C'})$ is an almost $\iota$-complex.  
\end{proof}

\begin{lem} \label{lem:deleting_summands_almost_iota}
	If an $\iota$-complex $\C$ splits into $\C' \oplus \mc{B}$, then $\C$ is locally equivalent to $\C'$. 
\end{lem}
\begin{proof}
	Let $i: C' \to C$ and $\pi : C \to C'$ be the natural inclusion and projection. It will be shown that $i$ and $\pi$ give a local equivalence. Clearly $i$ and $\pi$ are grading preserving chain maps. Since the generator of $H_*(C)$ lies in $C'$ and $\p C' \subset C'$, $i$ and $\pi$ give isomorphims on homology after localization by $U$. It only remains to show that $i$ and $\pi$ commute with $\biota$ up to chain homotopy. For $i$, $i \circ \biota|_{C'} = \biota \circ i$ since $\biota C' \subset C'$. Similarly for $\pi$, 
	$\pi \circ \biota = \biota|_{C'} \circ \pi$
	since $\biota C' \subset C'$. 
\end{proof}

\subsection{Precurves} \label{subsection:precurves}

Precurves provide a way of describing chain complexes geometrically as decorated immersed curves on a given surface, up to homotopy. The general construction is described in full detail in \cite{kotelskiy2019immersed} and \cite{Zibrowius_2020}. An \textit{arc system} $(\Sigma, \phi, A)$ is an oriented Morse cobordism $(\Sigma, \phi: \Sigma \to [0,1])$ of closed 1-manifolds and a set $A$ of 1 dimensional ascending manifolds with boundary in $\Sigma_1 = \phi^{-1}(1)$, with the requirement that
\begin{enumerate}[$\bullet$]
	\item
		The 1 dimensional ascending manifolds in $A$ are pairwise, disjoint and oriented, and
	\item
		$\Sigma \setminus \bigcup_{a \in A} N(a)$ consists of pairwise disjoint annuli such that each annulus bounds exactly one component of $\Sigma_0 = \phi^{-1}(0)$. Here $N(a)$ is a small tubular neighborhood of $a$. 
\end{enumerate}
The disjoint annuli in $\Sigma \setminus \bigcup_{a \in A} N(a)$ are called \textit{faces} and the set of faces is denoted $F(\Sigma, A)$. In the language of \cite{kotelskiy2019immersed} and \cite{Zibrowius_2020}, $\Sigma_0$ is the \textit{inner boundary} and $\Sigma_1$ is the \textit{outer boundary}. Let $s_2(a)$ and $s_1(a)$ be left and right sided boundaries of $N(a) - \p \Sigma$ respectively, according to the chosen orientation of $a$. An arc system will often be written as $(\Sigma, A)$ for brevity without the explicit Morse function $\phi$.

\begin{figure}[ht!]
	\begin{subfigure}{0.45 \textwidth}
		\def\svgwidth{\textwidth}
\begingroup%
  \makeatletter%
  \providecommand\color[2][]{%
    \errmessage{(Inkscape) Color is used for the text in Inkscape, but the package 'color.sty' is not loaded}%
    \renewcommand\color[2][]{}%
  }%
  \providecommand\transparent[1]{%
    \errmessage{(Inkscape) Transparency is used (non-zero) for the text in Inkscape, but the package 'transparent.sty' is not loaded}%
    \renewcommand\transparent[1]{}%
  }%
  \providecommand\rotatebox[2]{#2}%
  \newcommand*\fsize{\dimexpr\f@size pt\relax}%
  \newcommand*\lineheight[1]{\fontsize{\fsize}{#1\fsize}\selectfont}%
  \ifx\svgwidth\undefined%
    \setlength{\unitlength}{511.07099669bp}%
    \ifx\svgscale\undefined%
      \relax%
    \else%
      \setlength{\unitlength}{\unitlength * \real{\svgscale}}%
    \fi%
  \else%
    \setlength{\unitlength}{\svgwidth}%
  \fi%
  \global\let\svgwidth\undefined%
  \global\let\svgscale\undefined%
  \makeatother%
  \begin{picture}(1,0.5538651)%
    \lineheight{1}%
    \setlength\tabcolsep{0pt}%
    \put(0,0){\includegraphics[width=\unitlength,page=1]{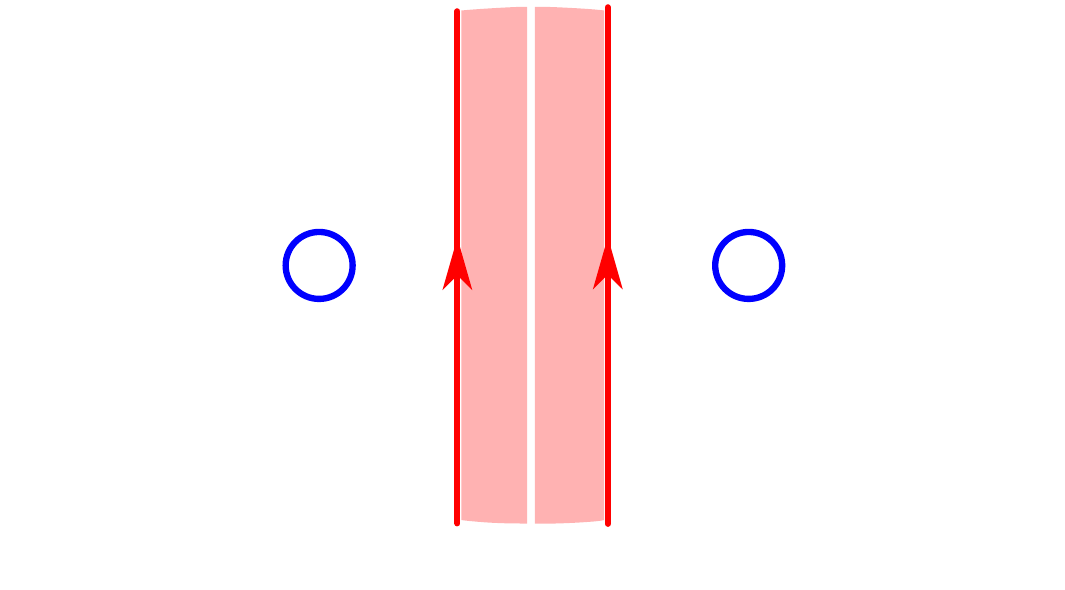}}%
    \put(0.47332727,0.00413882){\color[rgb]{0,0,0}\makebox(0,0)[lt]{\lineheight{1.25}\smash{\begin{tabular}[t]{l}$a$\end{tabular}}}}%
    \put(0.59703836,0.41122876){\color[rgb]{0,0,0}\makebox(0,0)[lt]{\lineheight{1.25}\smash{\begin{tabular}[t]{l}$s_1(a)$\end{tabular}}}}%
    \put(0.24794562,0.41193836){\color[rgb]{0,0,0}\makebox(0,0)[lt]{\lineheight{1.25}\smash{\begin{tabular}[t]{l}$s_2(a)$\end{tabular}}}}%
    \put(0,0){\includegraphics[width=\unitlength,page=2]{twice_punctured_arc_system.pdf}}%
    \put(0.9545532,0.29469959){\color[rgb]{0,0,0}\makebox(0,0)[lt]{\lineheight{1.25}\smash{\begin{tabular}[t]{l}$U$\end{tabular}}}}%
    \put(-0.00177706,0.29469959){\color[rgb]{0,0,0}\makebox(0,0)[lt]{\lineheight{1.25}\smash{\begin{tabular}[t]{l}$Q$\end{tabular}}}}%
  \end{picture}%
\endgroup%

	\end{subfigure}
	\hfill
	\begin{subfigure}{0.45 \textwidth}
		\def\svgwidth{\textwidth}
\begingroup%
  \makeatletter%
  \providecommand\color[2][]{%
    \errmessage{(Inkscape) Color is used for the text in Inkscape, but the package 'color.sty' is not loaded}%
    \renewcommand\color[2][]{}%
  }%
  \providecommand\transparent[1]{%
    \errmessage{(Inkscape) Transparency is used (non-zero) for the text in Inkscape, but the package 'transparent.sty' is not loaded}%
    \renewcommand\transparent[1]{}%
  }%
  \providecommand\rotatebox[2]{#2}%
  \newcommand*\fsize{\dimexpr\f@size pt\relax}%
  \newcommand*\lineheight[1]{\fontsize{\fsize}{#1\fsize}\selectfont}%
  \ifx\svgwidth\undefined%
    \setlength{\unitlength}{495.53531497bp}%
    \ifx\svgscale\undefined%
      \relax%
    \else%
      \setlength{\unitlength}{\unitlength * \real{\svgscale}}%
    \fi%
  \else%
    \setlength{\unitlength}{\svgwidth}%
  \fi%
  \global\let\svgwidth\undefined%
  \global\let\svgscale\undefined%
  \makeatother%
  \begin{picture}(1,0.77694981)%
    \lineheight{1}%
    \setlength\tabcolsep{0pt}%
    \put(0,0){\includegraphics[width=\unitlength,page=1]{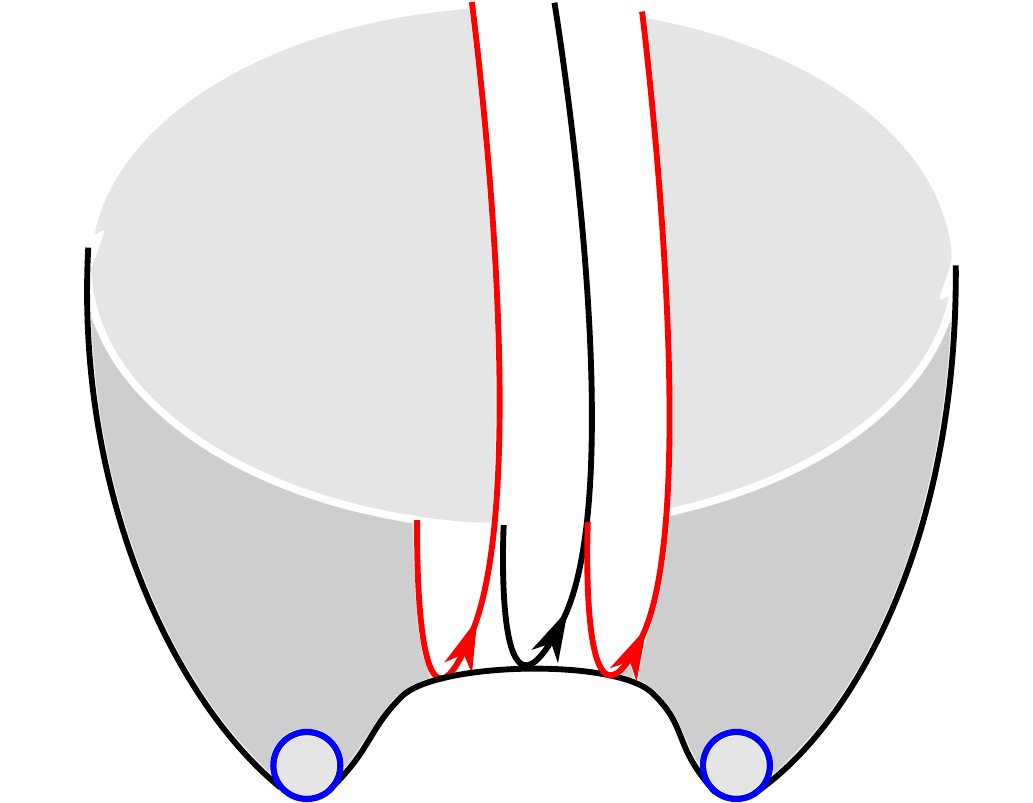}}%
    \put(0.48168019,0.07593475){\color[rgb]{0,0,0}\makebox(0,0)[lt]{\lineheight{1.25}\smash{\begin{tabular}[t]{l}$a$\end{tabular}}}}%
    \put(0.65943226,0.41254412){\color[rgb]{0,0,0}\makebox(0,0)[lt]{\lineheight{1.25}\smash{\begin{tabular}[t]{l}$s_1(a)$\end{tabular}}}}%
    \put(0.33615154,0.41111377){\color[rgb]{0,0,0}\makebox(0,0)[lt]{\lineheight{1.25}\smash{\begin{tabular}[t]{l}$s_2(a)$\end{tabular}}}}%
    \put(0,0){\includegraphics[width=\unitlength,page=2]{twice_punctured_disk_cobordism.pdf}}%
    \put(0.95312838,0.50857802){\color[rgb]{0,0,0}\makebox(0,0)[lt]{\lineheight{1.25}\smash{\begin{tabular}[t]{l}$U$\end{tabular}}}}%
    \put(-0.00183277,0.52155089){\color[rgb]{0,0,0}\makebox(0,0)[lt]{\lineheight{1.25}\smash{\begin{tabular}[t]{l}$Q$\end{tabular}}}}%
    \put(0,0){\includegraphics[width=\unitlength,page=3]{twice_punctured_disk_cobordism.pdf}}%
  \end{picture}%
\endgroup%

	\end{subfigure}
	\caption{An arc system on the twice punctured disk, shown from above and at an angle.}
	\label{fig:twice_punctured_disk_arc_system}
\end{figure}

A picture is worth a thousand words, of course. See figure~\ref{fig:twice_punctured_disk_arc_system} for an example of an arc system on the twice punctured disk. In this figure, the outer boundary $\Sigma_1$ is shown in green and the inner boundary $\Sigma_0$ is shown in blue.

There is an associated quiver algebra $Q(\Sigma, A)$ for any arc system $(\Sigma, A)$. $Q(\Sigma, A)$ is generated by the graph with vertices being the arcs $a \in A$ and the directed edges being the oriented components of $\Sigma_1 \setminus \bigcup_{a \in A} a$ which connect arc endpoints. See figure~\ref{fig:twice_punctured_disk_arc_algebra} for an example.

\begin{figure}[ht!]
	\def\svgwidth{0.4\textwidth}
\begingroup%
  \makeatletter%
  \providecommand\color[2][]{%
    \errmessage{(Inkscape) Color is used for the text in Inkscape, but the package 'color.sty' is not loaded}%
    \renewcommand\color[2][]{}%
  }%
  \providecommand\transparent[1]{%
    \errmessage{(Inkscape) Transparency is used (non-zero) for the text in Inkscape, but the package 'transparent.sty' is not loaded}%
    \renewcommand\transparent[1]{}%
  }%
  \providecommand\rotatebox[2]{#2}%
  \newcommand*\fsize{\dimexpr\f@size pt\relax}%
  \newcommand*\lineheight[1]{\fontsize{\fsize}{#1\fsize}\selectfont}%
  \ifx\svgwidth\undefined%
    \setlength{\unitlength}{420.26328283bp}%
    \ifx\svgscale\undefined%
      \relax%
    \else%
      \setlength{\unitlength}{\unitlength * \real{\svgscale}}%
    \fi%
  \else%
    \setlength{\unitlength}{\svgwidth}%
  \fi%
  \global\let\svgwidth\undefined%
  \global\let\svgscale\undefined%
  \makeatother%
  \begin{picture}(1,0.6267405)%
    \lineheight{1}%
    \setlength\tabcolsep{0pt}%
    \put(0,0){\includegraphics[width=\unitlength,page=1]{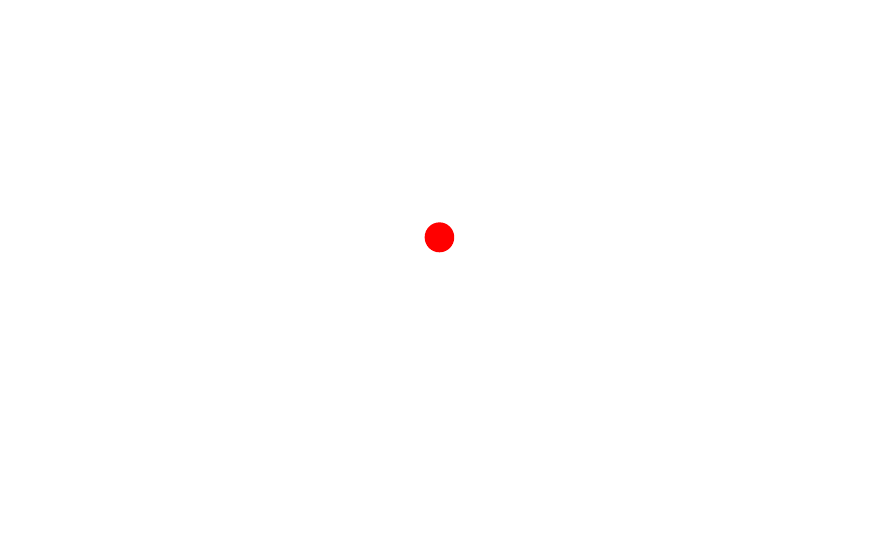}}%
    \put(0.44550598,0.00438019){\color[rgb]{0,0,0}\makebox(0,0)[lt]{\lineheight{1.25}\smash{\begin{tabular}[t]{l}$\iota_a$\end{tabular}}}}%
    \put(0,0){\includegraphics[width=\unitlength,page=2]{twice_punctured_arc_algebra.pdf}}%
    \put(0.15888721,0.00456606){\color[rgb]{0,0,0}\makebox(0,0)[lt]{\lineheight{1.25}\smash{\begin{tabular}[t]{l}$Q$\end{tabular}}}}%
    \put(0.76564975,0.00438019){\color[rgb]{0,0,0}\makebox(0,0)[lt]{\lineheight{1.25}\smash{\begin{tabular}[t]{l}$U$\end{tabular}}}}%
  \end{picture}%
\endgroup%

	\caption{The quiver algebra corresponding to the arc system shown in figure~\ref{fig:twice_punctured_disk_arc_system}.}
	\label{fig:twice_punctured_disk_arc_algebra}
\end{figure}

The algebra corresponding to an arc system is the \textit{arc algebra} 
$$\mc{A}(\Sigma, A) := Q(\Sigma, A)/\mathrm{ArcRelations}$$ 
where ArcRelations is the relation set
$$\mathrm{ArcRelations}: = \{ \rho_2 \rho_1 = \rho_3 \rho_4 = 0 : a \in A \}$$
and the $\rho_i$ are defined according to the local relationship shown in figure~\ref{fig:arc_algebra_quotient}. 

\begin{figure}[ht!]
	\def\svgwidth{0.4\textwidth}
\begingroup%
  \makeatletter%
  \providecommand\color[2][]{%
    \errmessage{(Inkscape) Color is used for the text in Inkscape, but the package 'color.sty' is not loaded}%
    \renewcommand\color[2][]{}%
  }%
  \providecommand\transparent[1]{%
    \errmessage{(Inkscape) Transparency is used (non-zero) for the text in Inkscape, but the package 'transparent.sty' is not loaded}%
    \renewcommand\transparent[1]{}%
  }%
  \providecommand\rotatebox[2]{#2}%
  \newcommand*\fsize{\dimexpr\f@size pt\relax}%
  \newcommand*\lineheight[1]{\fontsize{\fsize}{#1\fsize}\selectfont}%
  \ifx\svgwidth\undefined%
    \setlength{\unitlength}{408.73544393bp}%
    \ifx\svgscale\undefined%
      \relax%
    \else%
      \setlength{\unitlength}{\unitlength * \real{\svgscale}}%
    \fi%
  \else%
    \setlength{\unitlength}{\svgwidth}%
  \fi%
  \global\let\svgwidth\undefined%
  \global\let\svgscale\undefined%
  \makeatother%
  \begin{picture}(1,0.35155949)%
    \lineheight{1}%
    \setlength\tabcolsep{0pt}%
    \put(0,0){\includegraphics[width=\unitlength,page=1]{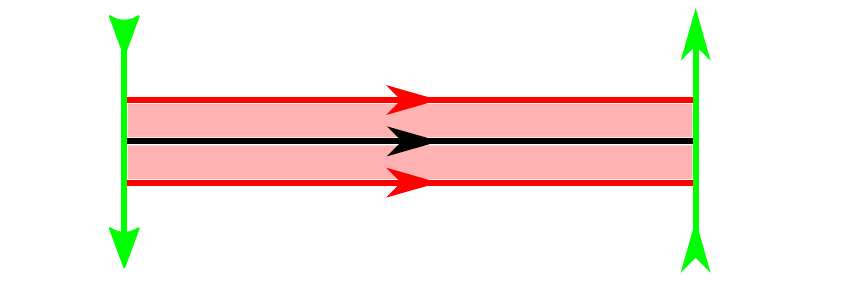}}%
    \put(0.42120934,0.31584099){\makebox(0,0)[lt]{\lineheight{1.25}\smash{\begin{tabular}[t]{l}$s_2(a)$\end{tabular}}}}%
    \put(0.41784542,0.03409827){\makebox(0,0)[lt]{\lineheight{1.25}\smash{\begin{tabular}[t]{l}$s_1(a)$\end{tabular}}}}%
    \put(-0.00222198,0.32896696){\makebox(0,0)[lt]{\lineheight{1.25}\smash{\begin{tabular}[t]{l}$\rho_1$\end{tabular}}}}%
    \put(0.8649609,0.32445259){\makebox(0,0)[lt]{\lineheight{1.25}\smash{\begin{tabular}[t]{l}$\rho_2$\end{tabular}}}}%
    \put(-0.00138744,0.00517507){\makebox(0,0)[lt]{\lineheight{1.25}\smash{\begin{tabular}[t]{l}$\rho_3$\end{tabular}}}}%
    \put(0.86365008,0.00517538){\makebox(0,0)[lt]{\lineheight{1.25}\smash{\begin{tabular}[t]{l}$\rho_4$\end{tabular}}}}%
    \put(0.07772842,0.1646103){\makebox(0,0)[lt]{\lineheight{1.25}\smash{\begin{tabular}[t]{l}$a$\end{tabular}}}}%
  \end{picture}%
\endgroup%

	\caption{To obtain $\A(\Sigma, A)$, $Q(\Sigma, A)$ is quotiented by $\rho_2 \rho_4 = \rho_3 \rho_1 = 0$ for each $a \in A$.}
	\label{fig:arc_algebra_quotient}
\end{figure}

The arc algebra corresponding to the arc system in figure~\ref{fig:twice_punctured_disk_arc_system} is $\mc{R}$. This is the arc system on which all the proofs in sections~\ref{section:classifying_almost_iota_complexes} and \ref{section:homomorphisms} will rely. By the end of this section, the reader will have the tools to describe a chain complex over $\mc{R}$ as an immersed curve on the twice punctured disk. For a demonstration, see example~\ref{example:simple_precurve_simplification}.

The key object in the translation of chain complexes to immersed curves is the precurve. Precurves are complicated objects, and their explanation will require some algebraic preliminaries. First it is necessary to define the linear extension $\Mat(\C)$ of a differential graded category $\C$. Compare the following to \cite{Zibrowius_2020} and \cite{barnatan}. In this presentation, we will only be working over rings with characteristic 2.
\begin{define}[\cite{Zibrowius_2020}, Definition 1.4; \cite{barnatan}, Definition 6.1]
	Given a differential graded category $\C$, define the linear extension of $\C$, $\Mat(\C)$ by setting
	$$\Ob = \{ \bigoplus_{k=0}^n x_k[i_k] : x_k \in \Ob(\C) \}$$
	where $x_k[i_k]$ is $x_k$ with grading shifted by $i_k$ and
	$$\Mor_i( \bigoplus_{k=0}^m x_k[i_k], \bigoplus_{l=0}^n x_l[i_l]) = \bigoplus_{(k,l) \in [0,m]\times[0,n]} \Mor_{i+ i_l - i_k}(x_k, x_l)$$
	where composition is given by matrix multiplication and $\p$ is defined in the natural fashion.
\end{define}
This allows for a definition of chain complexes over $\C$. 
\begin{define}[\cite{Zibrowius_2020}, Definition 1.5]
	Given a differential category $\C$, the category of complexes over $\C$, $\Cx(\C)$, is defined as follows:
	$$\Ob \Cx(\C) = \{(X, d) : X \in \Ob(\C), d \in \Mor(X, X; 1) , d^2 + \p(d) = 0 \}$$
	and
	$$\Mor_{\Cx(\C)}((X,d), (X', d')) = \Mor_{\C}(X, X').$$
	Define endomorphisms $D$ on these morphisms by
	$$D(f) = d' \circ f + f \circ d + \p f.$$
\end{define}
This provides a new, albeit more convoluted, way in which to understand chain complexes over a differential graded category. Before continuing, note $\A(\Sigma, A)$ can be made into a differential category, with 
\begin{enumerate}[$\bullet$]
	\item 
		the objects $\Ob(\A(\Sigma, A))$ being the idempotents $\mc{I} =  \{ \iota_a : a \in A \}$ in the arc algebra,
	\item
		the morphisms $\Mor_{\A(\Sigma, A)}(\iota_a, \iota_b)$ being $\iota_b . \A(\Sigma, A) . \iota_a$, and
	\item
		the differential being $\p = 0$. 
\end{enumerate}
Note that the morphisms were never explicitly given a grading. This will be discussed later in section~\ref{section:homomorphisms}. There are several ways to assign a valid grading, and the grading given in \cite{kotelskiy2020mnemonic} and \cite{Zibrowius_2020} will not be used here. Regardless, $\Cx(\Mat(\A(\Sigma, A)))$ is then the category of finitely and freely generated chain complexes over $\A(\Sigma, A)$, or type-$D$ structures over $\A(\Sigma, A)$. 

Consider for a moment the difficulty in simplifying chain complexes over $\mc{R}$ to some canonical form. The difficulty is in the fact that changing basis effects both the $U$-arrows and the $Q$-arrows, and it is challenging to treat the $U$-arrows and $Q$-arrows separately. See \cite{dai2018infiniterank} for another way to resolve this issue. The advantage of this complicated categorical construction is that it will provide an equivalent category in which the simplification of the $U$-arrows and $Q$-arrows can be separated. This equivalent category is the category of precurves. 

Before separating the category of chain complexes $\Cx(\Mat(\mc{R}))$ into $U$ and $Q$ components, the algebra $\mc{R}$ must first be separated. There is a general method for this over any arc system which takes the algebra $\A(\Sigma, A)$ to the expanded algebra $\bar{\A}(\Sigma, A)$. 
	
For each face $f \in F(\Sigma, A)$, there is an associated cyclic directed graph where the vertices are arcs $a \in \p f$, and the directed edges are the oriented boundary components of $\Sigma - \bigcup_{a \in A} a$ connecting the arcs. Let $\A_f$ be the path algebra of this cyclic graph. Let the \textit{length} of an element of $\A_f$ be the length of the path in the cyclic graph corresponding to that element.
Then, the expanded algebra is constructed by setting
$$\bar{\A}(\Sigma, A) := \bigoplus_{f \in F(\Sigma,A)} \A_f.$$ 
Denote $\bar{\mc{I}}$ to be the idempotents in $\bar{\A}(\Sigma, A)$. Note there are twice as many generating idempotents in $\bar{\A}(\Sigma, A)$ as there are in $\A(\Sigma, A)$. It will be important later to remember that $\mc{I}$ embeds into $\bar{\mc{I}}$ by $\iota_a \mapsto \iota_{s_1(a)} \oplus \iota_{s_2(a)}$. 
It will also be important to define the subalgebras with elements of length greater than zero in $\bar{\A}(\Sigma, A)$:
$$\bar{\A}^+(\Sigma, A) := \A / \{ \iota = 0 : \iota \in \bar{\mc{I}} \}.$$
Lastly, a special element of $\A_f(\Sigma, A)$ for each $f \in F(\Sigma, A)$, $U_f$, will be defined: $U_f \in \bar{\A}_f(\Sigma, A)$ is the sum of all algebra elements which are represented by a single directed edge around $f$.  

The case of the twice punctured disk is covered in example~\ref{example:extended_arc_algebra_of_twice_punctured_disk}.

\begin{example} \label{example:extended_arc_algebra_of_twice_punctured_disk}
	Consider the arc system from figure~\ref{fig:twice_punctured_disk_arc_system}. It has already been established that the arc algebra $\A(\Sigma, A)$ is $\mc{R}$. Here $\mc{I}$ is the single idempotent $\iota_a$, which for all intents and purposes can be considered to be 1. $\bar{\A}(\Sigma)$ is $\F[U] \oplus \F[Q]$ and $\bar{\mc{I}} = \{ \iota_{s_1(a)}, \iota_{s_2(a)} \}$. Here $\iota_{s_1(a)}$ corresponds to $0 \oplus 1$ and $\iota_{s_2(a)}$ corresponds to $1 \oplus 0$,  if it is taken that the first summand in $\bar{\A}(\Sigma, A)$ is generated by $U$ and the second by $Q$. $\mc{I}$ embeds in $\bar{\mc{I}}$ by $1 \mapsto 1 \oplus 1$. $\bar{\A}^+(\Sigma, A) = \{ \sum_{k=1}^n (\alpha_k U^k + \beta_k Q^k) : \alpha_k, \beta_k \in \F \}$. $U_f = U$ if $f$ contains the $U$-puncture, and $U_f= Q$ if $f$ contains the $Q$-puncture.
\end{example}

The subalgebra $\bar{\A}^+$ is important because it gives rise to an exact sequence:
$$0 \to \bar{\A}^+ \to \bar{\A} \to \bar{\mc{I}} \to 0$$
that splits the morphisms of the linear extensions:
$$0 \to \Mor_{\Mat \bar{\A}^+}(C, C') \to \Mor_{\Mat \bar{\A}}(C, C') \to \Mor_{\Mat \bar{\mc{I}}}(C, C') \to 0$$
thereby allowing any $\varphi \in \Mor_{\Mat \bar{\A}}(C, C')$ to be split into $\varphi^+ + \varphi^\times$ where $\varphi^+ \in \Mor_{\Mat \bar{\A}^+}(C, C')$ and $\varphi^\times \in \Mor_{\Mat \bar{\mc{I}}}(C, C')$. This splitting allows for the construction of an expanded linearization $\Mat_i \bar{\A}$. 

\begin{define}[\cite{Zibrowius_2020}, Definition 4.13]
	Let $\Mat_i \bar{\A}$ be the category defined by 
	\begin{enumerate}[$\bullet$]
		\item
			$\Ob \Mat_i \bar{\A} : = \{(C, (P_a)_{a \in A}) \}$
			where $C$ is a graded right module over $\bar{\mc{I}}$ and 
			$$P_a : C. \iota_{s_1(a)} \to C. \iota_{s_2(a)}$$
			is a grading preserving vector space isomorphism for every arc $a \in A$,
		\item
			$\Mor_{\Mat_i \bar{\A}} ((C, (P_a)), (C', (P'_a))) := $
			$$\{ \varphi \in \Mor_{\Mat \bar{\A}}(C, C') : (\iota_{s_2(a)}. \varphi^{\times} . \iota_{s_2(a)} ) \circ P_a = P_a' \circ (\iota_{s_1(a)}. \varphi^{\times} . \iota_{s_1(a)} ) \forall  a \in A \}$$
			where $\varphi^{\times}$ is the restricition of $\varphi$ to the identity component, as described in \cite{Zibrowius_2020} Definition 4.13, and
		\item
			composition in $\Mat_i \bar{\A}$ is inherited from $\Mat \bar{\A}$. 
	\end{enumerate}
\end{define}

With all this book-keeping in order, it is finally time to define the category of precurves.
\begin{define}[\cite{Zibrowius_2020}, Definition 4.15]
	$\Cx(\Mat_i\bar{\A})$ is the category of \textit{precurves}.
\end{define}

The following results were proven by Zibrowius in \cite{Zibrowius_2020}:

\begin{lem}[\cite{Zibrowius_2020}, Lemma 4.14] \label{lem:mat_a_and_mat_bar_a_are_equivalent}
	$\Mat \A$ and $\Mat_i \bar{\A}$ are equivalent categories.
\end{lem}
\begin{proof}
	This is witnessed by the equivalences $\mc{F} : \Mat \A \to \Mat_i \bar \A$ and $\mc{G} : \Mat_i \bar \A \to \Mat \A$. The functor $\mc{F}: \Mat \A \to \Mat_i \bar{\A}$ 
	\begin{enumerate}[$\bullet$]
		\item 
		takes $\iota_a \in \Ob(\A) = \mc{I}$ (idempotents) to $\iota_{s_1(a)} \oplus \iota_{s_2(a)}$ in $\Ob(\bar \A) = \bar{\mc{I}}$ (also idempotents),
		\item
		takes morphisms in $\varphi \in \Mor_{\A}(\iota_a, \iota_b)$ to 
		\begin{center}
			\begin{tikzcd}[ampersand replacement=\&]
			\iota_{s_1(a)} \oplus \iota_{s_2(a)} 
			\arrow{rrrrr}{
				\begin{pmatrix}
				\iota_{s_1(b)} . \varphi . \iota_{s_1(a)} & \iota_{s_1(b)} . \varphi . \iota_{s_2(a)} \\ 
				\iota_{s_2(b)} . \varphi . \iota_{s_1(a)} & \iota_{s_2(b)} . \varphi . \iota_{s_2(a)}
				\end{pmatrix}
			} \& \& \& \& \& 
			\iota_{s_1(b)} \oplus \iota_{s_2(b)}
			\end{tikzcd}
		\end{center}
		in $\Mor_{\Mat_i \bar \A} (\iota_{s_1(a)} \oplus \iota_{s_2(a)}, \iota_{s_1(b)} \oplus \iota_{s_2(b)})$
		\item
		and for each $x. \iota_a \in C.\iota_a$, $P_a(x. \iota_{s_1(a)}) = x.\iota_{s_2(a)}$. 
	\end{enumerate}
	
	$\mc{G}: \Mat_i \bar{\A} \to \Mat \A$ is roughly the projection of $C.\iota_{s_2(a)} \oplus C. \iota_{s_1(a)}$ to $C.\iota_a \cong C.\iota_{s_1(a)}$. The morphisms must begin at an element of $C. \iota_{s_1(a)}$ and end at an element $C.\iota_{s_1(b)}$ for any $b$ to make this projection into an equivalence. In full formality, the functor $\mc{G}$ 
	
	\begin{enumerate}[$\bullet$]
		\item 
		takes an object $X = (C, \{P_a\}_{a \in A})$ to $\mc{G}(X) = \bigcup_{a \in A} C.\iota_{s_1(a)}$ where $\mc{G}(X).\iota_a = C.\iota_{s_1(a)}$,
		\item
		takes a morphism $\varphi \in \Mor_{\Mat_i \bar \A}((C, \{P_a\}), (C', \{ P'_a\}))$ to $\mc{G}(\varphi)$ in $\Mat \A$ defined by
		\begin{equation} \label{equation:mcG_on_morphisms}
		\begin{aligned}
		\iota_{b}.\mc{G}(\varphi).\iota_a  & =  (\iota_{s_1(b)}.\varphi^{\times}.\iota_{s_1(a)}) + (\iota_{s_1(b)}.\varphi^+.\iota_{s_1(a)}) + (P'_b)^{-1} \circ (\iota_{s_2(b)}.\varphi^{+}.\iota_{s_1(a)}) \\
& + (\iota_{s_1(b)}.\varphi^{+}.\iota_{s_2(a)})  \circ (P_a) + (P'_b)^{-1} \circ (\iota_{s_2(b)}. \varphi^+ . \iota_{s_2(a)})\circ P_a
		\end{aligned}
		\end{equation}
		for each $a,b \in A$.
	\end{enumerate}

	It is a quick calculation to show that $\mc{F}$ and $\mc{G}$ compose to the identity.
\end{proof}

An immediate corollary of Lemma~\ref{lem:mat_a_and_mat_bar_a_are_equivalent} is that the category of chain complexes over $\A$ is equivalent to the category of precurves.

\begin{cor}[\cite{Zibrowius_2020}, Corollary 4.16] \label{cor:cxmata_and_cxmatibara_are_equivalent}
	$\Cx(\Mat \A)$ are $\Cx(\Mat_i \bar{\A})$ equivalent differential categories. 
\end{cor}
\begin{proof}
	Consider the functors $\mc{F}$ and $\mc{G}$ from Lemma~\ref{lem:mat_a_and_mat_bar_a_are_equivalent}. Then $\mc{F}$ and $\mc{G}$ are mutually inverse and witness the equivalence of $\Cx(\Mat \A)$ and $\Cx(\Mat_i\bar{\A})$. 
\end{proof}

The benefit of considering chain complexes over $\A(\Sigma, A)$ as precurves is that precurves can be represented geometrically as decorated immersed curves on $\Sigma$ up to homotopy. This geometric description will be described shortly. When considering the geometric presentation of a precurve, it will be easiest to work with \textit{reduced} precurves.
\begin{define}
	A \textit{reduced} precurve $(C, \{ P_a\}, d)$ is a precurve such that $d^+ = d$. 
\end{define}
In other words, a reduced precurve is a complex in which the differential has no idempotents. Zibrowius proves in \cite{Zibrowius_2020} that precurves are reduced up to homotopy.
\begin{prop}[\cite{Zibrowius_2020}, Lemma 4.17]
	All precurves are chain homotopic to a reduced one. 
\end{prop}

\begin{construction} \label{construction:drawing_reduced_precurves}
	Now to describe the geometric presentation of an arbitrary precurve $(C, \{P_a\}, \p) \in \Cx \Mat_i \bar \A(\Sigma, A)$ on $\Sigma$. The first step is to draw $C$. For each $a$, fix a basis of $C.\iota_{s_1(a)}$ and $C. \iota_{s_2(a)}$,
	$$\{ e_1^{s_1(a)}, \cdots, e_{n_a}^{s_1(a)} \} \qquad \text{ and } \qquad  \{e_1^{s_2(a)}, \cdots, e_{n_a}^{s_2(a)} \}. $$
	Start by marking a point on $s_1(a)$ and $s_2(a)$ for the first generators of $C.\iota_{s_1(a)}$ and $C.\iota_{s_2(a)}$, respectively. Following the orientation of $a$ and the order of each basis, mark a point on $s_1(a)$ and $s_2(a)$.
	
	The next step is to draw the differential. Consider each nonzero component of the differential:
	\begin{center}
		\begin{tikzcd}
		e_i^{s_k(a)} \arrow[r,"p^+"] & e_j^{s_l(a)}.
		\end{tikzcd}
	\end{center}
	since $\p$ is reduced, $p^+ \in \bar{\A}^+$ is a nontrivial path about some face $f \in F(\Sigma, A)$. For each such nonzero component, draw a curve from $e_i^{s_k(a)}$ to $e_j^{s_l(a)}$, following the path $p^+$. If, at the end of this process, there is some generator not connected to a path, connect it to the puncture on the face which it neighbors.
	
	The final step is to draw $\{P_a\}$. Every matrix $P_a \in \mathrm{GL}_n(\F = \Z_2)$ can be written in the form $P_a = P_a^{l_a} \cdots P_a^1$ where each matrix $P_a$ is a transposition $T_{ij}$ (transposing columns $i$ and $j$), or a single row addition $E_{ij}$ (adding row $i$ to row $j$). $P_a$ can then be represented within $N(a)$ in the spirit of \cite{hanselman2016bordered}, in which $T_{ij}$ gives a \textit{crossing} and $E_{ij}$ gives a \textit{crossover arrow}. This method is shown in figure~\ref{fig:p_a_example}. 
	\begin{figure}[ht!]
		\def\svgwidth{0.5\textwidth}
		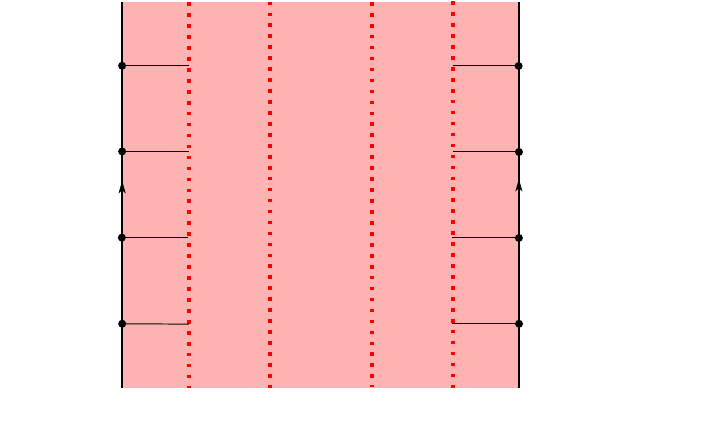
		\caption{Representing $P_a$ in $N(a)$.}
		\label{fig:p_a_example}
	\end{figure}
\end{construction}

There is no longer any need to distinguish between the geometric and algebraic representations of a reduced precurve. 
Note that it is unreasonable to expect a reduced precurve to be an immersed curve in general. At the very least, it is possible there are many components of $\p$ associated with a single generator. Denote a precurve in which each generator is associated with exactly one component as \textit{simply-faced}. The following lemma is required to manipulate an arbitrary immersed curve into a simply-faced curve:

\begin{lem}[Clean-Up Lemma, \cite{Zibrowius_2020}, Lemma 1.24] \label{lem:clean_up}
	Let $(C, d_C)$ be an object in $\Cx(\Mat \C)$ for some differential category $\C$. Then for any morphism $h \in \Mor((C, d_C), (C, d_C)$ for which
	$$h^2, \qquad hD(h), \qquad \text{ and } \qquad D(h) h$$
	are all zero, $(C, d_C)$ is chain isomorphic to $(C, d_C + D(h))$. 
\end{lem}
It is straightforward to check that $(C, d_C + D(h))$ is in $\Cx(\Mat(\C))$ and that $(1+h) : (C, d_C) \to (C, d_C + D(h))$ and $(1+h) : (C, d_C + D(h)) \to (C, d_C)$ are chain isomorphisms. For a detailed presentation, see \cite{Zibrowius_2020}. Using the clean-up lemma it is possible to prove the following:

\begin{prop}[\cite{Zibrowius_2020}, Proposition 4.24] \label{prop:reduced_chain_iso_to_simply_faced}
	Each reduced precurve is chain isomorphic to a simply-faced precurve. 
\end{prop}

This is shown by inductively applying the \textit{arc-reduction} moves shown in figure~\ref{fig:precurve_simplification_arcs}. In figure~\ref{fig:precurve_simplification_arcs}, the integers $m$ and $n$ are the length of the arc. The clean-up lemma can be applied to show that each arc-reduction move gives a chain isomorphism. The chain isomorphisms induced by these moves S1a, S1b, S2a and S2b are implicitly described in the proof of \ref{prop:reduced_chain_iso_to_simply_faced} in \cite{Zibrowius_2020}, but it will be helpful in section~\ref{section:homomorphisms} to have them written down. This is the content of Lemma~\ref{lem:s1_s_2_iso}.

\begin{figure}[ht!]
	\def\svgwidth{0.9\textwidth}
	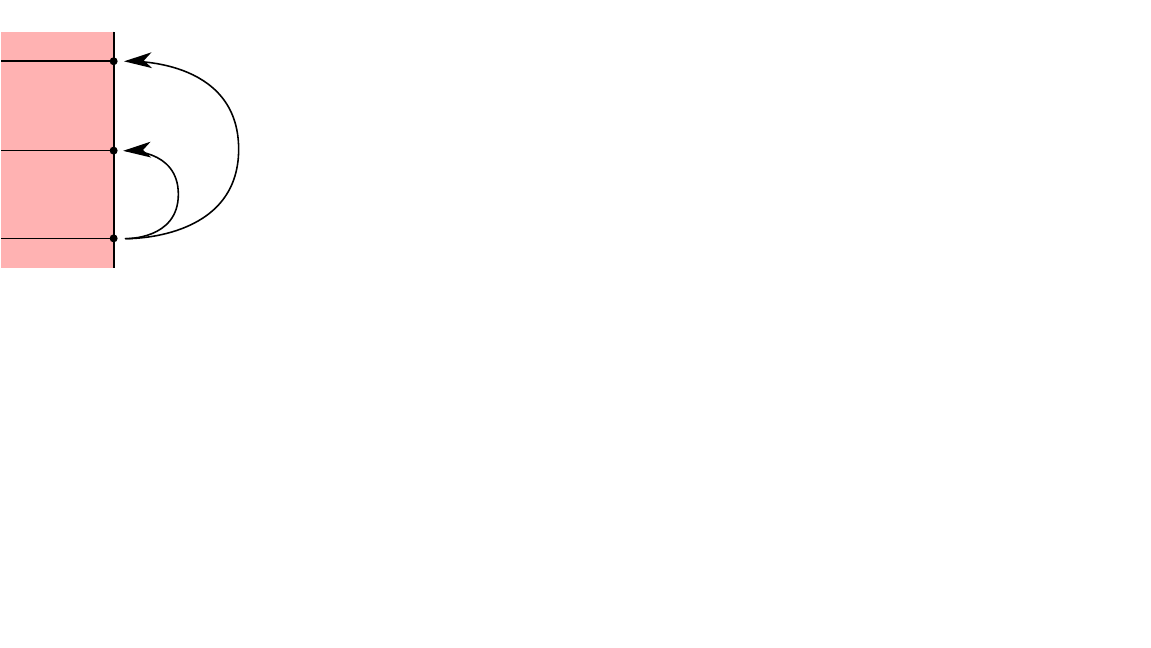
	\caption{The arc-reductions, S1a, S1b, S2a, and S2b. Here it be can assumed that $m,n > 0$. The arcs are shown with orientation so that faces do not have to be drawn. Note that the orientation is generated by the direction of the path algebra element in the differential.}
	\label{fig:precurve_simplification_arcs}
\end{figure}

\begin{lem} \label{lem:s1_s_2_iso}
	The arc-reductions S1 and S2 are precurve chain isomorphisms.
\end{lem}
\begin{proof}
	First consider S1a. If the precurve on the LHS of S1a in figure~\ref{fig:precurve_simplification_arcs} is $((C, \{P_a \}), d)$, then the complex on the RHS of S1a in figure~\ref{fig:precurve_simplification_arcs} is $((C, \{ P_a\}), d + D(h))$ where 
	$$
	h(p) = \begin{cases}
	(\iota_a.U_f^{m-n}.\iota_a) x & \text{ if } p = y \\
	0 & \text{ else }
	\end{cases}
	$$
	Furthermore, $h \in \Mor_{\Mat_i \bar{\A}}((C, \{P_a\}), (C, \{ P_a\}))$ since $h^\times = 0$ so
	$$(\iota_{s_2(a)}. h^{\times} . \iota_{s_2(a)} ) \circ P_a = P_a' \circ (\iota_{s_1(a)}. h^{\times} . \iota_{s_1(a)} ).$$
	Lemma~\ref{lem:clean_up} gives that 
	$$(1+h): ((C,\{P_a\}), d) \to ((C,\{P_a\}), d + D(h)) \text{ and } $$ 
	$$(1+h) : ((C,\{P_a\}), d + D(h)) \to ((C,\{P_a\}),d)$$ 
	are both chain isomorphisms since $h^2 = h D(h) = D(h) h = 0$. The proof for S2a is identical, except $h$ is defined by
	$$
	h(p) = 
	\begin{cases}
	(\iota_a.U_f^{m-n}.\iota_a)  y & \text{ if } p = x \\
	0 & \text{ else } 	
	\end{cases}.
	$$
	Now consider S1b. Here 
	$$ 
	h(p) = 
	\begin{cases}
	x & \text{ if } p = y \\
	0 & \text{ else } 	
	\end{cases}.
	$$
	It could be that $h \not \in \Mor_{\Mat_i \bar{\A}}((C, \{P_a\}), (C, \{ P_a\}))$ since $h = h^{\times} \neq 0$ and it may not be that
	$$(\iota_{s_2(a)}. h^{\times} . \iota_{s_2(a)} ) \circ P_a = P_a' \circ (\iota_{s_1(a)}. h^{\times} . \iota_{s_1(a)} ).$$
	Let $a$ be the arc on which $x,y,$ and $z$ lie. If $P_a$ is pre- or post-compose with $E_{yx}$ to get $P_a'$ then $h \in \Mor_{\Mat_i \bar{\A}}((C, \{P_a\}), (C, \{ P_a'\}))$. Thus the clean-up lemma gives that 
	$$(1+h): ((C,\{P_a\}), d) \to ((C,\{P_a'\}), d + D(h)) \text{ and }$$
	$$(1+h): ((C,\{P_a'\}), d+ D(h)) \to ((C,\{P_a\}),d)$$ 
	are both chain isomorphisms. The proof for S2b is identical, except $P_a$ is pre- or post-composed with $E_{xy}$ and 
	$$h(p) = 
	\begin{cases}
	y & \text{ if } p = x \\
	0 & \text{ else } 	
	\end{cases}.
	$$
\end{proof}

Note that even if a precurve is simply-faced, it still may not be a collection of decorated immerersed curves, since it is unclear how to treat the crossover arrows and crossings in each arc neighborhood. 

\begin{define}
	A \textit{decorated immersed curve} is a pair $(\gamma, X)$ where $\gamma$ is either closed or $\p \gamma \subset \Sigma_0$, and $X \in GL_n(\F_2)$. 
\end{define}

\begin{construction} \label{construction:drawing_decorations}
	A decorated immersed curve $(\gamma, X)$ is realized as a precurve by the following process:

	\begin{enumerate}[1. ]
		\item
		place $\dim X$ parallel copies of $\gamma$, $\{ \gamma_1, \cdots, \gamma_{\dim X}\}$, in a small neighborhood of eachother, then
		\item 
		pick any arc $a \in A$ intersecting $\gamma$, and index the generators $\gamma_i \cap s_j(a) = e_i^{s_j(a)}$, and finally
		\item
		represent $X$ as crossover arrows and crossings between the $\gamma_i$, according to the same method which $P_a$ is represented on $N(a)$. 
	\end{enumerate}
\end{construction}

\begin{define}
	A \textit{multicurve} is a finite set of decorated immersed curves.
\end{define}

Now for the punchline:

\begin{thm} \label{thm:class_decorated_immersed}
	Any reduced precurve is chain isomorphic to a multicurve 
	$$\{ (\gamma_1, X_1), \cdots, (\gamma_n, X_n) \}.$$ 
\end{thm}

The algorithm which takes any reduced precurve to a set of decorated immersed curves is known as the \textit{arrow-sliding algorithm} and it is described in Section 3.6 of \cite{hanselman2016bordered}.
The arrow-sliding algorithm does not use the moves S1 and S2 from figure~\ref{fig:precurve_simplification_arcs}. Instead, it uses the new moves T1, T2, T3, M1, M2, and M3 shown in figure~\ref{fig:precurve_simplification_p_a}. Since \cite{hanselman2016bordered} describes the arrow-sliding algorithm in a slightly different context, some translation is required. This is the content of the proof of Theorem~\ref{thm:class_decorated_immersed} in \cite{Zibrowius_2020}. 

This translation relies on two concepts: the construction of an ordering $\leq$ on strands in $N(a)$, and the assignment of a depth to crossover arrows. Before describing this translation, it is useful to define strands on a precurve:
\begin{define} \label{def:strand}
	Given a graphical presentation of a precurve and an arc $a$, a strand in $N(a)$ is a path connecting generators in $C.\iota_{s_1(a)}$ and $C.\iota_{s_2(a)}$ which follows crossings in the graphical reprentation of $P_a$.
\end{define}

Take two strands $s, t$ in $N(a)$. Place an imaginary crossover arrow from strand $s$ to $t$, and ignore all crossover arrows in the precurve. Push that imaginary arrow along each strand in the direction from $a$ to $s_1(a)$. If the curves never diverge, then they are parallel. In this case, $s \leq t$ and $t \leq s$. If the curves diverge and M3a, M3b, or M3c can push that imaginary arrow out, then $s \leq t$. If not, then $s > t$. 

Any simply-faced precurve can be homotoped such that the ordering of the strands with respect to $\leq$ and the orientation of $a$ agree. This may require adding new crossings in $P_a$ to the precurve by applying M1. The resulting precurve is an \textit{ordered} simply-faced precurve.

Crossover arrows in an ordered simply-faced precurve can be assigned a depth. Let $E_{ij}$ pass between strands $s$ and $t$. The depth of a crossover arrow $E_{ij}$ is the minimum number of arc neighborhoods traversed following $s$ and $t$ in either direction until a divergence occurs. This depth can be infinite if $s$ and $t$ are parallel. 

From here, the arrow-sliding algorithm of \cite{hanselman2016bordered} directly applies T1, T2, T3, M1, M2 and M3 to iteratively increase the minimum depth of any arrow in the precurve. Since the paths in the precurve are only of finite length, this process results in a set of decorated immersed curves.

In the language of \cite{hanselman2016bordered}, M1 is a \textit{crossing slide}, while M2 and M3 are \textit{crossover arrow slides}. 
It will be particularly helpful in the proof of Theorems~\ref{thm:p_omega_homo} and \ref{thm:p_homo} to know the explicit chain isomorphisms involved, so they are demonstrated here. 

\begin{lem} \label{lem:t1_t2_t3_iso}
	The moves T1, T2, and T3 are chain isomorphisms of precurves.
\end{lem}
\begin{proof}
	The moves T1, T2, and T3 only change the graphical representation of the precurve, not the algebraic structure of the precurve itself. Thus the chain isomorphism induced by T1, T2, and T3 is the identity 
\end{proof}

\begin{lem} \label{lem:m1_iso}
	The crossing slide M1 is a chain isomorphism of precurves.
\end{lem}
\begin{proof}
	Let the LHS of M1 in figure~\ref{fig:precurve_simplification_p_a} be $((C, \{P_a\}), d)$ and the RHS of M1 in figure~\ref{fig:precurve_simplification_p_a} be $((C, \{P_a'\}), d')$. Let $f: C \to C$ be the map interchanging $x$ and $y$. Then 
	$$f  \in \Mor_{\Mat_i \bar{\A}}((C, \{P_a\}), (C, \{P_a'\}))$$
	since 
	$$ (\iota_{s_2(a)}. f. \iota_{s_2(a)} ) \circ P_a = P_a' \circ (\iota_{s_1(a)}. f. \iota_{s_1(a)} ).$$
	Since $fd = d'f$ it follows that $f$ is a chain map and
	$$f: ((C, \{P_a\}), d) \to (C, \{P_a'\}), d') \text{ and }$$
	$$f: ((C, \{P_a'\}, d') \to (C, \{P_a\}), d)$$
	are chain isomorphisms. 
\end{proof}

\begin{lem} \label{lem:m2_iso}
	The crossover arrow slide M2 is a chain isomorphism of precurves.
\end{lem}
\begin{proof}
	See the drawing of M2 in figure~\ref{fig:precurve_simplification_p_a}.
	Let $h_1$ take the top left-most generator to the bottom left-most generator in figure~\ref{fig:precurve_simplification_p_a}, and let $h_2$ take the top right-most generator to the bottom right-most generator in figure~\ref{fig:precurve_simplification_p_a}. The chain isomorphism $(1+h_1)$ performs the reverse action of S1b and S2b from Lemma~\ref{lem:s1_s_2_iso}, and pushes the crossover arrow out of the left-most arc neighborhood. The chain isomorphism $(1+h_2)$ pushes an arrow into the right-most arc neighborhood. Applying $(1+h_2)$ is the same as applying either S1b or S2b from Lemma~\ref{lem:s1_s_2_iso}. Thus the compositve isomorphism is $(1+h_2) \circ (1+h_1)$. But $(1+h_2) \circ (1+h_1) = 1 + h_1 + h_2$, so the arrow slide M2 corresponds to the chain isomorphism $1+h_1 + h_2$.
\end{proof}

\begin{lem} \label{lem:m3_iso}
	The crossover arrow slide M3 is a chain isomorphism of precurves.
\end{lem}
\begin{proof} See figure~\ref{fig:precurve_simplification_p_a}. Let the LHS of M3a, M3b, and M3c be denoted $(C, \{P_a \}, d)$ and the RHS of M3a, M3b, and M3c be denoted $(C, \{P_a'\}, d)$.
Define $h$ to take the top generator to the bottom generator in each representation of M3 in figure~\ref{fig:precurve_simplification_p_a}. The clean-up lemma gives that $(1+h)$ is a chain isomorphism. For details, see the proof of \ref{lem:s1_s_2_iso}, which is identical.
\end{proof}

\begin{figure}[p]
	\begin{subfigure}{\textwidth}
		\def\svgwidth{0.9\textwidth}
\begingroup%
  \makeatletter%
  \providecommand\color[2][]{%
    \errmessage{(Inkscape) Color is used for the text in Inkscape, but the package 'color.sty' is not loaded}%
    \renewcommand\color[2][]{}%
  }%
  \providecommand\transparent[1]{%
    \errmessage{(Inkscape) Transparency is used (non-zero) for the text in Inkscape, but the package 'transparent.sty' is not loaded}%
    \renewcommand\transparent[1]{}%
  }%
  \providecommand\rotatebox[2]{#2}%
  \newcommand*\fsize{\dimexpr\f@size pt\relax}%
  \newcommand*\lineheight[1]{\fontsize{\fsize}{#1\fsize}\selectfont}%
  \ifx\svgwidth\undefined%
    \setlength{\unitlength}{575.7855777bp}%
    \ifx\svgscale\undefined%
      \relax%
    \else%
      \setlength{\unitlength}{\unitlength * \real{\svgscale}}%
    \fi%
  \else%
    \setlength{\unitlength}{\svgwidth}%
  \fi%
  \global\let\svgwidth\undefined%
  \global\let\svgscale\undefined%
  \makeatother%
  \begin{picture}(1,0.51336015)%
    \lineheight{1}%
    \setlength\tabcolsep{0pt}%
    \put(0,0){\includegraphics[width=\unitlength,page=1]{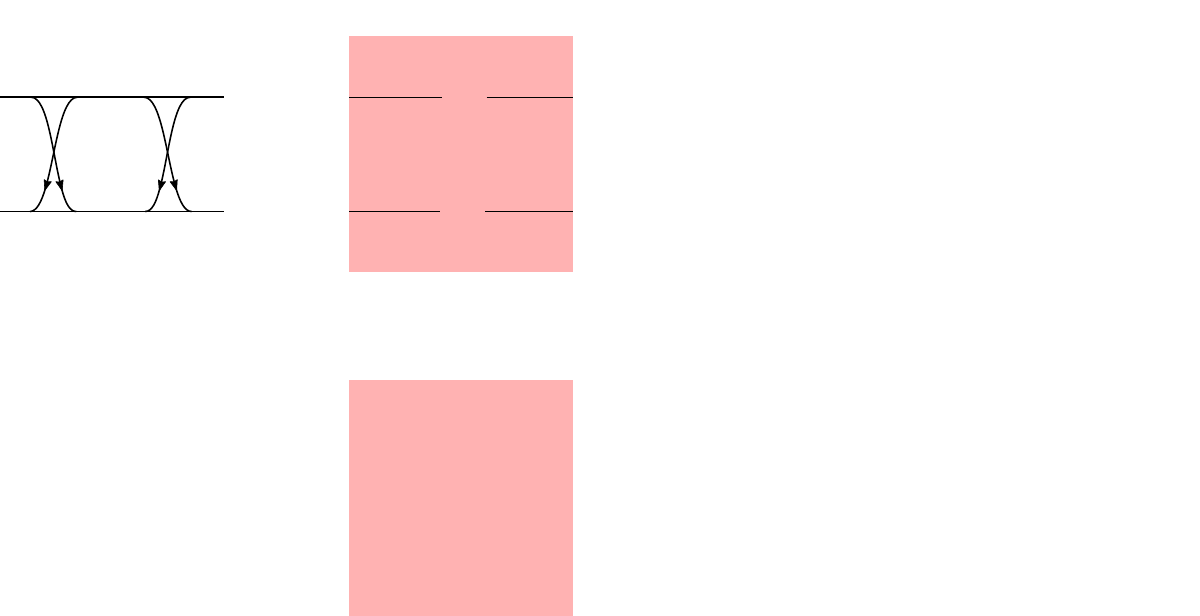}}%
    \put(-0.00113558,0.21343021){\makebox(0,0)[lt]{\lineheight{1.25}\smash{\begin{tabular}[t]{l}$\mathrm{T2}:$\end{tabular}}}}%
    \put(0,0){\includegraphics[width=\unitlength,page=2]{T1_T2_T3.pdf}}%
    \put(-0.00113559,0.49999529){\makebox(0,0)[lt]{\lineheight{1.25}\smash{\begin{tabular}[t]{l}$\mathrm{T1}:$\end{tabular}}}}%
    \put(0,0){\includegraphics[width=\unitlength,page=3]{T1_T2_T3.pdf}}%
    \put(0.51943921,0.40623066){\makebox(0,0)[lt]{\lineheight{1.25}\smash{\begin{tabular}[t]{l}$\mathrm{T3}:$\end{tabular}}}}%
    \put(0,0){\includegraphics[width=\unitlength,page=4]{T1_T2_T3.pdf}}%
  \end{picture}%
\endgroup%

		\caption{The moves T1, T2, and T3, which change the graphical representation of $P_a$. }
	\end{subfigure}
	\begin{subfigure}{\textwidth}
		\def\svgwidth{0.9\textwidth}
		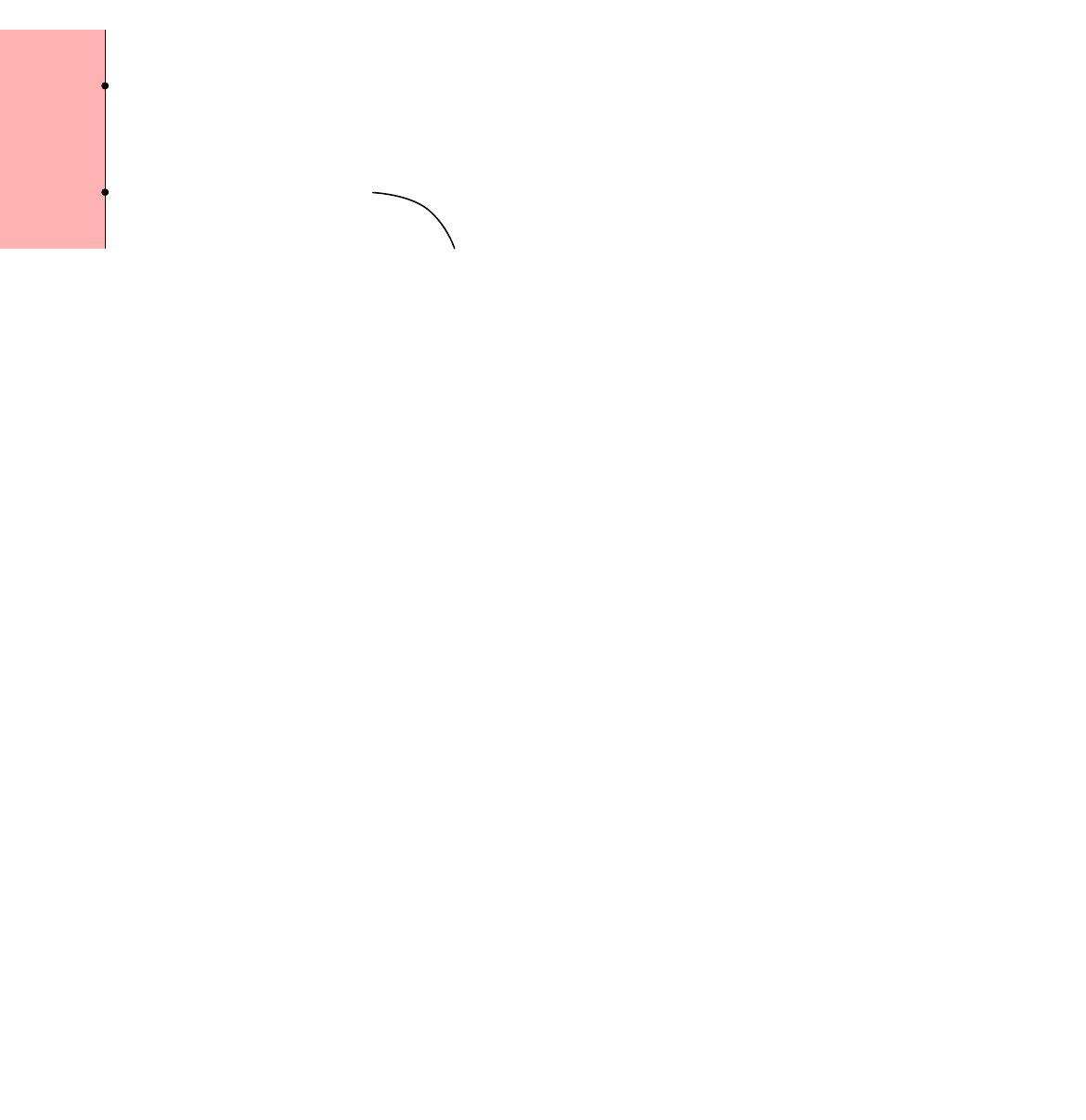
		\caption{The moves M1, M2, M3a, M3b, and M3c, which add crossings and crossover arrows.}
	\end{subfigure}
	\caption{All moves necessary to reduce a simply-faced precurve to a decorated immersed curve.}
	\label{fig:precurve_simplification_p_a}
\end{figure}

\begin{example} \label{example:simple_precurve_simplification}
	Now to give an example of the precurve simplification process in action. Consider the $\mc{R}$ complex shown below:
	\begin{center}
		\begin{tikzcd}[sep=small]
			& & \bullet_y \arrow[rrd, red, "Q"] & &	\\
			\bullet_x \arrow[rru, "U"] \arrow[rrd, "U"] &  & & & \bullet_z \\
			& & \bullet_w \arrow[rru, "Q", red] &  & 
		\end{tikzcd}
	\end{center}
	The precurve representation of this complex is shown in figure~\ref{fig:simple_precurve_example}. The simplification of this precurve to an immersed curve is shown in figure~\ref{fig:simple_precurve_example} as well. 
	\begin{figure}[p]
		\begin{subfigure}{0.45\textwidth}
			\def\svgwidth{0.9\textwidth}
			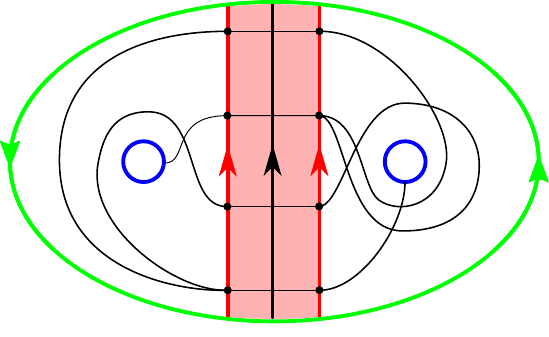
		\end{subfigure}
		\hfill
		\begin{subfigure}{0.45\textwidth}
			\begin{center}
				\begin{tikzcd}[sep=small]
				& & \bullet_y \arrow[rrd, red, "Q"] & &	\\
				\bullet_x \arrow[rru, "U"] \arrow[rrd, "U"] &  & & & \bullet_z \\
				& & \bullet_w \arrow[rru, "Q", red] &  & 
				\end{tikzcd}
			\end{center}
		\end{subfigure} \\
		\begin{subfigure}{0.45\textwidth}
			\def\svgwidth{0.9\textwidth}
\begingroup%
  \makeatletter%
  \providecommand\color[2][]{%
    \errmessage{(Inkscape) Color is used for the text in Inkscape, but the package 'color.sty' is not loaded}%
    \renewcommand\color[2][]{}%
  }%
  \providecommand\transparent[1]{%
    \errmessage{(Inkscape) Transparency is used (non-zero) for the text in Inkscape, but the package 'transparent.sty' is not loaded}%
    \renewcommand\transparent[1]{}%
  }%
  \providecommand\rotatebox[2]{#2}%
  \newcommand*\fsize{\dimexpr\f@size pt\relax}%
  \newcommand*\lineheight[1]{\fontsize{\fsize}{#1\fsize}\selectfont}%
  \ifx\svgwidth\undefined%
    \setlength{\unitlength}{263.40278339bp}%
    \ifx\svgscale\undefined%
      \relax%
    \else%
      \setlength{\unitlength}{\unitlength * \real{\svgscale}}%
    \fi%
  \else%
    \setlength{\unitlength}{\svgwidth}%
  \fi%
  \global\let\svgwidth\undefined%
  \global\let\svgscale\undefined%
  \makeatother%
  \begin{picture}(1,0.7217293)%
    \lineheight{1}%
    \setlength\tabcolsep{0pt}%
    \put(0,0){\includegraphics[width=\unitlength,page=1]{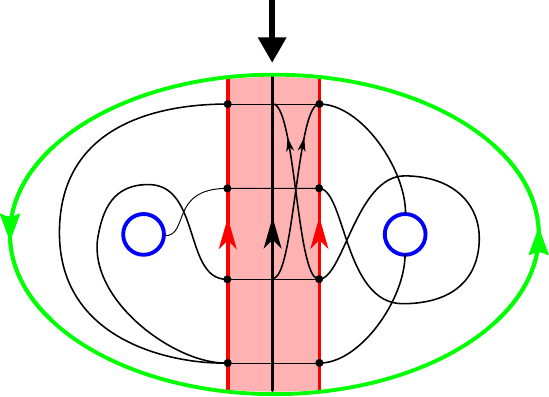}}%
    \put(0.5410418,0.64999329){\makebox(0,0)[lt]{\lineheight{1.25}\smash{\begin{tabular}[t]{l}S1b\end{tabular}}}}%
    \put(0,0){\includegraphics[width=\unitlength,page=2]{simple_precurve_example_1.pdf}}%
  \end{picture}%
\endgroup%

		\end{subfigure}
		\hfill
		\begin{subfigure}{0.45\textwidth}
			\begin{center}
				\begin{tikzcd}[sep=small]
				& & \bullet \arrow[rrd, red, "Q"] & &	\\
				\bullet \arrow[rru, "U"] &  & & & \bullet \\
				& & \bullet \arrow[rru, "Q", red] &  & 
				\end{tikzcd}
			\end{center}
		\end{subfigure} \\
		\begin{subfigure}{0.45\textwidth}
			\def\svgwidth{0.9\textwidth}
\begingroup%
  \makeatletter%
  \providecommand\color[2][]{%
    \errmessage{(Inkscape) Color is used for the text in Inkscape, but the package 'color.sty' is not loaded}%
    \renewcommand\color[2][]{}%
  }%
  \providecommand\transparent[1]{%
    \errmessage{(Inkscape) Transparency is used (non-zero) for the text in Inkscape, but the package 'transparent.sty' is not loaded}%
    \renewcommand\transparent[1]{}%
  }%
  \providecommand\rotatebox[2]{#2}%
  \newcommand*\fsize{\dimexpr\f@size pt\relax}%
  \newcommand*\lineheight[1]{\fontsize{\fsize}{#1\fsize}\selectfont}%
  \ifx\svgwidth\undefined%
    \setlength{\unitlength}{263.40278622bp}%
    \ifx\svgscale\undefined%
      \relax%
    \else%
      \setlength{\unitlength}{\unitlength * \real{\svgscale}}%
    \fi%
  \else%
    \setlength{\unitlength}{\svgwidth}%
  \fi%
  \global\let\svgwidth\undefined%
  \global\let\svgscale\undefined%
  \makeatother%
  \begin{picture}(1,0.72309321)%
    \lineheight{1}%
    \setlength\tabcolsep{0pt}%
    \put(0,0){\includegraphics[width=\unitlength,page=1]{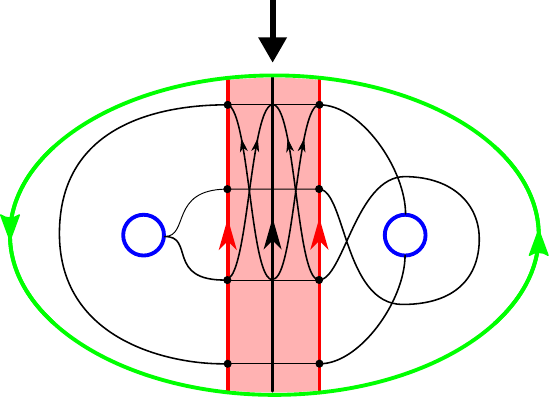}}%
    \put(0.54125708,0.65593307){\makebox(0,0)[lt]{\lineheight{1.25}\smash{\begin{tabular}[t]{l}S2b\end{tabular}}}}%
    \put(0,0){\includegraphics[width=\unitlength,page=2]{simple_precurve_example_2.pdf}}%
  \end{picture}%
\endgroup%

		\end{subfigure}
		\hfill
		\begin{subfigure}{0.45\textwidth}
			\begin{center}
				\begin{tikzcd}[sep=small]
				& & \bullet  & &	\\
				\bullet \arrow[rru, "U"] &  & & & \bullet \\
				& & \bullet \arrow[rru, "Q", red] &  & 
				\end{tikzcd}
			\end{center}
		\end{subfigure} \\
		\begin{subfigure}{0.45\textwidth}
			\def\svgwidth{0.9\textwidth}
\begingroup%
  \makeatletter%
  \providecommand\color[2][]{%
    \errmessage{(Inkscape) Color is used for the text in Inkscape, but the package 'color.sty' is not loaded}%
    \renewcommand\color[2][]{}%
  }%
  \providecommand\transparent[1]{%
    \errmessage{(Inkscape) Transparency is used (non-zero) for the text in Inkscape, but the package 'transparent.sty' is not loaded}%
    \renewcommand\transparent[1]{}%
  }%
  \providecommand\rotatebox[2]{#2}%
  \newcommand*\fsize{\dimexpr\f@size pt\relax}%
  \newcommand*\lineheight[1]{\fontsize{\fsize}{#1\fsize}\selectfont}%
  \ifx\svgwidth\undefined%
    \setlength{\unitlength}{263.40279756bp}%
    \ifx\svgscale\undefined%
      \relax%
    \else%
      \setlength{\unitlength}{\unitlength * \real{\svgscale}}%
    \fi%
  \else%
    \setlength{\unitlength}{\svgwidth}%
  \fi%
  \global\let\svgwidth\undefined%
  \global\let\svgscale\undefined%
  \makeatother%
  \begin{picture}(1,0.72172957)%
    \lineheight{1}%
    \setlength\tabcolsep{0pt}%
    \put(0,0){\includegraphics[width=\unitlength,page=1]{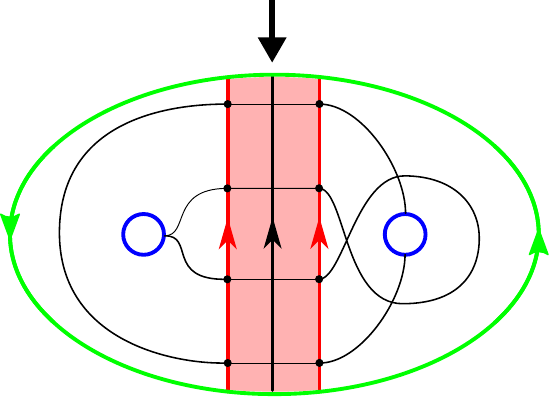}}%
    \put(0.53883117,0.65008824){\makebox(0,0)[lt]{\lineheight{1.25}\smash{\begin{tabular}[t]{l}T2\end{tabular}}}}%
    \put(0,0){\includegraphics[width=\unitlength,page=2]{simple_precurve_example_3.pdf}}%
  \end{picture}%
\endgroup%

		\end{subfigure}
		\hfill
		\begin{subfigure}{0.45\textwidth}
			\begin{center}
				\begin{tikzcd}[sep=small]
				& & \bullet & &	\\
				\bullet \arrow[rru, "U"]  &  & & & \bullet \\
				& & \bullet \arrow[rru, "Q", red] &  & 
				\end{tikzcd}
			\end{center}
		\end{subfigure} \\
		\begin{subfigure}{0.45\textwidth}
			\def\svgwidth{0.9\textwidth}
			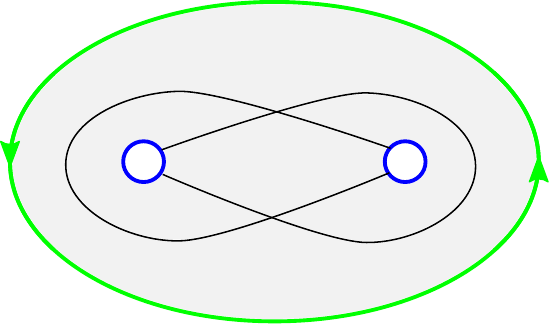
		\end{subfigure}
		\hfill
		\caption{An example of a precurve being simplified to two immersed curves. On the left hand side is the precurve, and on the right hand side is the $\mc{R}$ complex, found by applying the functor $\mc{G}$ to the precurve. At the bottom is the final immersed curve.}
		\label{fig:simple_precurve_example}
	\end{figure}
\end{example}

\section{Classifying almost $\iota$-complexes} \label{section:classifying_almost_iota_complexes}

This section realizes the classification of almost $\iota$-complexes up to local equivalence as a classification problem of complexes over $\mc{R} := \F[U,Q]/(UQ)$. We will show use the fact that almost $\iota$-complexes can be seen as curves on the twice punctured disk.

\subsection{Realizing almost $\iota$-complexes as immersed curves} \label{subsection_almost_iota_complexes_as_immersed_curves}

Recall from Section~\ref{section:prelims} that any almost $\iota$-complex is $\mathcal{C} = (C, \partial, \biota)$ is composed of a $\Z$-graded chain complex $(C, \partial)$ over $\F[U]$ and an approximate involution $\biota: C \to \C$ for which $\biota^2 \simeq \id \mod U$. Let the generators of $C$ be $\{ x_1, \cdots, x_n \}$. Denote the $\F$-span of these generators by $X$. We can interpret this almost $\iota$-complex $\C$ as a type D structure over $\mc{R}/(Q^2) = \F[U,Q]/(UQ, Q^2)$ which we will denote $\CI$. The elements of $\CI$ are $\mc{R} / (Q^2) \otimes_{\F} X$ and the differential is given by $\delta^1_{\CI} : X \to \mc{R}/(Q^2) \otimes_{\F} X$ where
$$\delta^1_{\CI} : = \p + Q \omega : X \to \mc{R}/(Q^2) \otimes_{\F} X.$$
\begin{define} \label{def:r_mod_q_2_representative_of_c}
	The \textit{$\mc{R}/(Q^2)$ representative of $\C$} is $\CI$ for any almost $\iota$-complex $\C$.
\end{define}
Recall from Lemma 3.23 of \cite{dai2018infiniterank} that $\omega^2 \equiv 0 \mod 0$ when $\C$ is reduced. It then follows that $\CI$ can be lifted to a type D structure over $\mc{R}$, $\tI{\C}$. The elements of $\tCI$ are $\mc{R} \otimes X$, and the differential is given by
$$\delta^1_{\tCI} : = \p + Q \omega : X \to \mc{R} \otimes_{\F} X.$$
The differential $\delta^1_{\tCI}$ squares to zero, as
$$(\delta^1_{\tCI})^2 x = \delta^1_{\tCI}(\p x + Q \omega x) = \p^2 x + Q \omega \p x + \p Q \omega x  + Q^2 \omega^2 x = Q [ \p, \omega] x + Q^2 \omega^2 = 0.$$
Note that $\omega^2 \equiv 0$ modulo $U$ when $\C$ is reduced, and that $[\p, \omega]$ is always 0 mod $U$ for any almost $\iota$-complex.
\begin{define} \label{def:r_representative_of_c}
	The \textit{$\mc{R}$ representative of $\C$} is $\tCI$ for any reduced almost $\iota$-complex $\C$. 
\end{define}
Note also that any type D structure $M$ over $\mc{R}$ can be made into a type D structure over $\mc{R}/(Q^2)$ by setting $Q^2$ to zero.  
\begin{define}
	The $Q^2$-\textit{reduction} of a type D structure over $\mc{R}$ is the type D structure over $\mc{R}/(Q^2)$ obtained by setting $Q^2$ to zero. 
\end{define}
The following proposition follows immediately from the definitions of $\tCI$ and $\CI$:
\begin{prop}
	The $Q^2$-reduction of $\tCI$ is $\CI$. 
\end{prop}
Up to local equivalence, it can be assumed that any almost $\iota$-complex has a lift to a type D structure over $\mc{R}$. 
\begin{cor} \label{prop:almost_iota_complexes_as_q2_reduction}
	Every almost $\iota$-complex is locally equivalent to the $Q^2$-reduction of a type D structure over $\mc{R}$.
\end{cor}
\begin{proof}
	Take an almost $\iota$-complex $\C$. By Lemma~\ref{lem:almost_iota_locally_equivalent_to_reduced}, $\C$ is locally equivalent to some reduced almost $\iota$-complex $\C_{r}$. Then the $Q^2$-reduction of $\tI{\C_{r}}$ is locally equivalent to $\C$. 
\end{proof}
The advantage of working over $\mc{R}$ is that any type D structure over $\mc{R}$ can be represented as a precurve on the twice punctured disk. Recall the functors $\mc{F} : \Cx \Mat(\A) \to \Cx \Mat_i(\bar{\A})$ and $\mc{G} : \Cx \Mat_i(\bar{\A}) \to \Cx \Mat (\A)$ used in the proof of Lemma~\ref{lem:mat_a_and_mat_bar_a_are_equivalent} and Corollary~\ref{cor:cxmata_and_cxmatibara_are_equivalent}.
\begin{define} \label{def:precurve_representative_of_c}
	The \textit{precurve representative of $\C$} is $\mc{F}(\tCI)$ for any reduced almost $\iota$-complex $\C$.
\end{define}

It will be useful to define the $Q^2$-reduction of a precurve. 

\begin{define}
	The $Q^2$-\textit{reduction} of a precurve $M$ on the twice punctured disk is the $Q^2$-reduction of $\mc{G}(M)$.
\end{define}

It is evident that the $Q^2$-reduction of $M(\tCI)$ is $\C$ for any reduced almost $\iota$-complex $\C$. 

\begin{prop} \label{prop:reduced_homotopy_gives_precurve_chain_iso}
	If $\C_1$ and $\C_2$ are two reduced almost $\iota$-complexes and $\C_1$ is homotopic to $\C_2$ then $\mc{F}(\tI{\C_1})$ is chain isomorphic to $\mc{F}(\tI{\C_2})$.
\end{prop}
\begin{proof}
	Suppose $f: \C_1 \to \C_2$ and $g: \C_2 \to \C_1$ witness the homotopy equivalence of $\C_1$ and $\C_2$ as almost $\iota$-complexes. Here $f$ and $g$ will be considered as $\F[U]$-module homomorphisms. Then $f$ and $g$ are almost $\iota$-morphisms between reduced complexes, so $[\omega, f]$ and $[\omega, g]$ are equivalent to $0 \mod U$ by Lemma~\ref{lem:reduced_has_w_f_commute_mod_U}. Furthermore, since $f$ and $g$ are homotopy equivalences of $\F[U]$ chain complexes, $f_*:H_*(\C_1) \to H_*(\C_2)$ and $g_*:H_*(\C_2) \to H_*(\C_1)$ must be isomorphisms. However, the homology of an $\F[U]$ chain complex completely determines its chain isomorphism type, so $f$ and $g$ must be chain isomorphisms.
	
	Now let the lifts of $f$ and $g$ be $\tilde{f} := f \otimes \id_{\mc{R}}$ and $\tilde{g} := g \otimes \id_{\mc{R}}$. Then 
	$$[\delta_1, \tilde{f}] = [\p + Q \omega, \tilde{f}] = [\p, \tilde{f}] + Q [\omega, \tilde{f}] = 0$$
	and
	$$ [\delta_1, \tilde{g}] = [\p + Q \omega, \tilde{g}] = [\p, \tilde{g}] + Q [\omega, \tilde{g}] = 0$$
	since $\tilde{f}$ and $\tilde{g}$ are chain maps with respect to $\p$, and both $[\omega, \tilde{f}]$ and $[\omega, \tilde{g}]$ are 0 mod $U$. Now to show that $\tilde{f}$ and $\tilde{g}$ are chain isomorphisms. Since $f$ and $g$ are chain isomorphisms, $\id_{\F[U]} + fg = 0$ and $\id_{\F[U]} + gf = 0$. Thus 
	$$\id_{\mc{R}} + \tilde{f} \tilde{g} = \id_{\F[U]} \otimes \id_{\mc{R}} + (f \otimes \id_{\mc{R}})(g \otimes \id_{\mc{R}}) = 
	(\id_{F[U]} + fg) \otimes \id_{\mc{R}} = 0$$
	and
	$$\id_{\mc{R}} + \tilde{g} \tilde{f} = \id_{\F[U]} \otimes \id_{\mc{R}} + (g \otimes \id_{\mc{R}})(f \otimes \id_{\mc{R}}) = 
	(\id_{F[U]} + gf) \otimes \id_{\mc{R}} = 0$$
	so  $\tilde{f}$ and $\tilde{g}$ are chain isomorphisms as well. Applying $\mc{F}$ gives the chain isomorphisms $\mc{F}(\tilde{f}):\mc{F}(\tI{\C_1}) \to \mc{F}(\tI{\C_2})$ and $\mc{F}(\tilde{g}): \mc{F}(\tI{\C_2}) \to \mc{F}(\tI{\C_1})$ between $\mc{F}(\tI{\C_1})$ and $\mc{F}(\tI{\C_2})$. 
\end{proof}

However, if $\C_1$ and $\C_2$ are two locally equivalent reduced almost $\iota$-complexes, then it is not true that $\mc{F}(\tI{\C_1})$ is chain isomorphic, or even homotopic, to $\mc{F}(\tI{\C_2})$. See example~\ref{example:local_equiv_counterexample}.

\begin{example} \label{example:local_equiv_counterexample}
	 In this example $\C$ and $\C'$ are locally equivalent, but the extra summand of $\C'$ adds a second immersed curve. These cannot be homotopic, since $H_*(\tI{\C_1}) \not \cong H_*(\tI{\C_2})$. See figure~\ref{fig:local_equivalence_counter_example} below.
	\begin{figure}[ht!]
		\begin{subfigure}{0.45\textwidth}
			\begin{center}
				\begin{tikzcd}
				\bullet & \bullet \arrow[r, "2"] \arrow[l, red] & \bullet
				\end{tikzcd}
			\end{center}
			\def \svgwidth{\textwidth}
			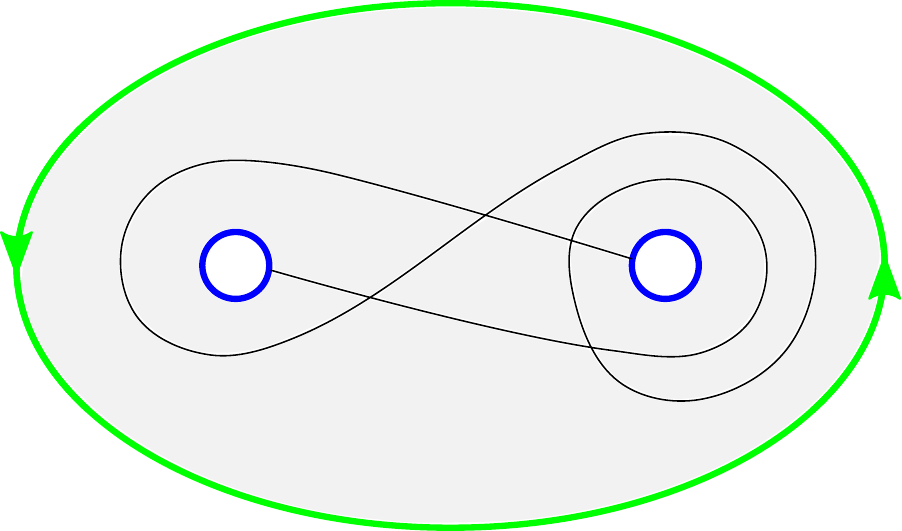
			\caption{$\C_1 = \C(+, -2)$ as an immersed curve.}
			\label{fig:local_equiv_counter_example_a}
		\end{subfigure}
		\hfill
		\begin{subfigure}{0.45\textwidth}
			\begin{center}
				\begin{tikzcd}
				\bullet & \bullet \arrow[r, "2"] \arrow[l, red] & \bullet & \bullet \arrow[r, "1"] & \bullet
				\end{tikzcd}
			\end{center}
			\def \svgwidth{\textwidth}
			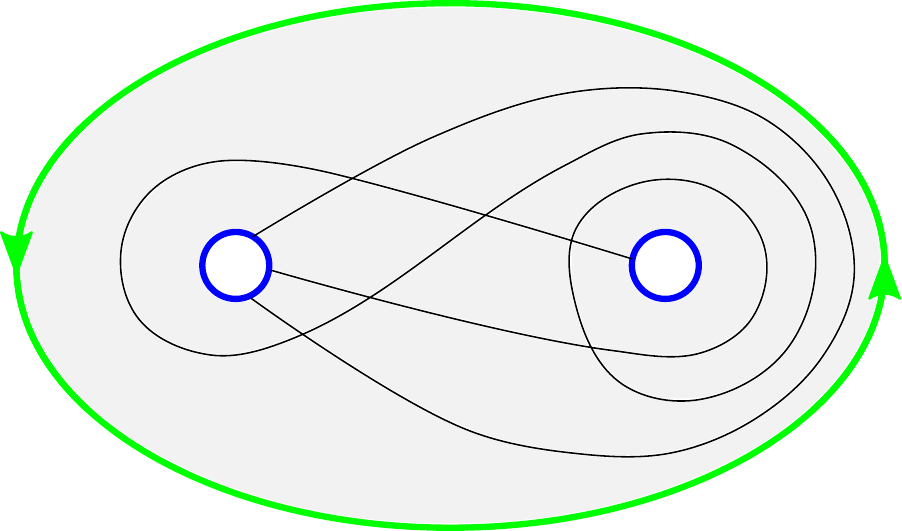
			\caption{$\C_2 = \C(+, -2) \oplus (\bullet \overset{U}{\to} \bullet)$ as an immersed curve.}
			\label{fig:local_equiv_counter_example_b}
		\end{subfigure}
		\caption{$\C_1$ is locally equivalent to $\C_1$, but $\mc{F}(\tI{\C_1})$ is not homotopic to $\mc{F}(\tI{\C_2})$.}
		\label{fig:local_equivalence_counter_example}
	\end{figure}
\end{example}

In sections~\ref{subsection_primitive_representatives} and \ref{subsection_classification}, it will be shown how to extract the local equivalence class of a reduced almost $\iota$-complex $\C$ from $\mc{F}(\tCI)$. Before doing so, it will be necessary to reduce $\mc{F}(\tCI)$ to a decorated immersed curve on the twice punctured disk, which is possible by Theorem~\ref{thm:class_decorated_immersed}. 

\begin{define} \label{def:multicurve_representative}
	The \textit{multicurve representative of $\C$} is the multicurve $\gamma(\tCI) = \{ (\gamma_i, X_i) \}$ which is chain isomorphic to $\mc{F}(\tCI)$ for any almost $\iota$-complex $\C$. 
\end{define}

We conclude section~\ref{subsection_almost_iota_complexes_as_immersed_curves} with the following lemma:

\begin{lem} \label{lem:one_q_loop}
	For a reduced almost $\iota$-complex $\C$, any differential component of $\gamma(\tCI)$ which passes around the $Q$-puncture does so exactly once.  
\end{lem}
\begin{proof}
	In $\mc{F}(\tCI)$, the only power of $Q$ appearing in $\delta^1$ is 1. However, $\mc{F}(\tCI)$ is not necessarily a decorated immersed curve. Consider the moves S1-3, M1-3 and T1-3 shown in figures~\ref{fig:precurve_simplification_arcs} and \ref{fig:precurve_simplification_p_a}, which simplify a precurve to a set of decorated immersed curves. These moves do not alter the number of times a segment of $\mc{F}(\tCI)$ wraps around the $Q$-puncture, unless it removes the segment entirely. Thus any segment of $\gamma(\tCI)$ that passes around the $Q$-puncture does so exactly once.
\end{proof}

\subsection{Quotienting to complexes over $\F[U]$ and $\F[Q]$} \label{subsection:constructing_complexes_over_fu_and_fq}
Recall that for any precurve $M$, $\mc{G}(M)$ is a type D structure over $\mc{R}$ with generators $X$ and
$$\delta^1 = \mc{G}(d): X \to \mc{R} \otimes X.$$
This defines a chain complex $C(M; \mc{R})$ over $\mc{R}$ which is equal to $\mc{G}(M)$.
By quotienting $U$ and $Q$, new type D structures over $\F[U]$ and $\F[Q]$ with the same generators can be defined in the natural way. Denote these new type D structures as 
$$(X, \delta^1/(Q) := \delta^1 \otimes \id_{\F[U,Q]/(Q)})$$
and 
$$(X, \delta^1/(U) :=\delta^1 \otimes \id_{\F[U,Q]/(U)}).$$
These define chain complexes $C(M; \F[U])$ and $C(M; \F[Q])$ over $\F[U]$ and $\F[Q]$ corresponding to $M$.
In this context, the grading of the above type D structures is ignored. 

While $H_*(M; \mc{R})$ is difficult to understand, $H_*(M; \F[U])$ and $H_*(M; \F[Q])$ are more manageable. 

In this section and the next, it will be useful to have a simplified notation for generators on a precurve and generators in the $\mc{R}$ complex corresponding to that precurve. 
Let the generators $C.\iota_{s_1(a)}$ be denoted $\{x_1^1, \cdots, x_n^1\}$ and their counterparts in $C.\iota_{s_2(a)}$ be denoted $\{x_1^2, \cdots, x_n^2\}$. The convention will be that $\mc{G}(x^1_i) = x_i$ and $\mc{G}(x^2_i) = 0$. The action of $\mc{G}$ on morphisms is given in equation~\ref{equation:mcG_on_morphisms}.
\begin{lem} \label{lem:rank_of_homology_over_fu_and_fq_from_counting_puncture_connections}
	For any reduced precurve on the twice punctured disk,
	$$\rk H_*(M; \F[U]) = \{ \text{number of generators connecting to the $U$-puncture}\}.$$
	and
	$$\rk H_*(M; \F[Q]) = \{ \text{number of generators connecting to the $Q$-puncture}\}.$$
\end{lem}
\begin{proof}
	Let the precurve $M$ be $(C, \{P_a\}, d)$. 
	In homology over $\F[U]$ this is slightly more straightforward to show since $s_1(a)$ is taken to be on the side of the $U$-puncture and $\mc{G}$ projects onto $C.\iota_{s_1(a)}$. Following the convention of the previous section, let $C.\iota_{s_1(a)}= \{ x_1^1, \cdots, x_n^1\}$, $C.\iota_{s_2(a)} = \{ x_1^2, \cdots, x_n^2\}$, $\mc{G}(x_i^1) = x_i$ and $\mc{G}(x_i^2) = 0$. Then by equation~\ref{equation:mcG_on_morphisms}
	\begin{align*}
	\delta^1 = \mc{G}(d) & =  (\iota_{s_1(b)}.d.\iota_{s_1(a)}) + (P'_b)^{-1} \circ (\iota_{s_2(b)}.d.\iota_{s_1(a)}) \\
	& + (\iota_{s_1(b)}.d.\iota_{s_2(a)})  \circ (P_a) + (P'_b)^{-1} \circ (\iota_{s_2(b)}.d. \iota_{s_2(a)})\circ P_a.
	\end{align*}
	Since $M$ is reduced, $d^\times = 0$ and $d^+ = d$. On the twice punctured disk, there are no algebra elements between $s_1(a)$ and $s_2(a)$ and there is only one arc $a$, so
	\begin{align*}
	\delta^1 & =  (\iota_{s_1(a)}.d.\iota_{s_1(a)}) + (P'_a)^{-1} \circ (\iota_{s_2(a)}.d. \iota_{s_2(a)})\circ P_a.
	\end{align*}
	Note that $\iota_{s_2(a)}.d.\iota_{s_2(a)}$ will only produce powers of $Q$ in the differential, so 
	$$\delta^1/(Q) = \iota_{s_1(a)}.d.\iota_{s_1(a)}.$$ 
	Similarly, $\iota_{s_1(a)}.d.\iota_{s_1(a)}$ will only produce powers of $U$ in the differential, so
	$$\delta^1/(U) = P_a^{-1} \circ ( \iota_{s_2(a)}.d.\iota_{s_2(a)}) \circ  P_a.$$ 
	Thus if there is a component connecting $x^1_i$ to the $U$-puncture, then $x_i \in \ker (\delta^1/(Q))$ and $U^k x_i \not \in \im (\delta^1/(Q))$ for any $k$. If $x_i$ connects to another $x_j$ by an arc, then either $x_i \not \in \ker (\delta^1/(Q))$ or there is some power of $k$ for which $U^k x_i \in \im (\delta^1/(Q))$. It then follows that
	$$U^{-1} H_*(M; \F[U]) = \F[U, U^{-1}] \langle \{[x_i] : x_i^1 \text{ connects to the $U$-puncture}\} \rangle,$$
	and then that
	$$\rk H_*(M; \F[U]) = \{ \text{number of generators connecting to the $U$-puncture}\}.$$
	Now it remains to do the same for homology over $\F[Q]$. Recall $\delta^1/(U) = P_a^{-1} \circ (\iota_{s_2(a)}. d .\iota_{s_2(a)}) \circ P_a$. Thus $x_i \in \ker (\delta^1/(Q))$ if $d P_a(x_i^1) = 0$, and $Q^k x_i \not \in \im (\delta^1/(Q))$ if there $P_a(x_i^1) \not \in d(C.\iota_{s_2(a)})$. The only elements of $C.\iota_{s_2(a)}$ that are in the kernel of $d$ and have no power of $Q$ in the image of $d$ are exactly the generators on $s_2(a)$ connecting to the $Q$-puncture. Thus
	$$Q^{-1} H_*(M; \F[Q]) = \F[Q, Q^{-1}] \langle [\mc{G} P_a^{-1} x_i^2] : x_i^2 \text{ connects to the $Q$-puncture} \} \rangle,$$
	and
	$$\rk H_*(M; \F[Q]) = \{ \text{number of generators connecting to the $Q$-puncture}\}$$
	follows, completing the proof. 
\end{proof}

This last lemma takes chain isomorphisms between $C(M; \mc{R}) \to C(M'; \mc{R})$ to chain isomorphisms $C(M; \F[U]) \to C(M'; \F[U])$ and $C(M; \F[Q]) \to C(M; \F[Q])$, and it will be useful in determining where a precurve connects to the $U$- or $Q$-punctures.

\begin{lem} \label{lem:r_chain_iso_descent}
	Suppose there is a chain isomorphism $\Phi : C(M; \mc{R}) \to C(M'; \mc{R})$. Then $\Phi$ descends to chain isomorphisms
	$$\Phi_U : C(M; \F[U]) \to C(M'; \F[U]) \qquad \text{ and } \qquad \Phi_Q: C(M; \F[Q]) \to C(M'; \F[Q]).$$
	where $\Phi_U = \Phi \otimes \id_{\F[U,Q]/(Q)}$ and $\Phi_Q = \Phi \otimes \id_{\F[U,Q]/(U)}$. 
\end{lem}
\begin{proof}
	Since $\Phi$ is a chain map, $\Phi \delta^1_{\C} = \delta^1_{\C'} \Phi$. $\Phi_U$ is a chain map because
	\begin{align*}
	(\Phi \delta^1_\C )\otimes \id_{\F[U,Q]/(Q)} & = (\delta^1_{\C'} \Phi) \otimes \id_{\F[U,Q]/(Q)} \\
	(\Phi \otimes \id_{\F[U,Q]/(Q)}) (\delta^1_{\C} \otimes \id_{\F[U,Q]/(Q)}) & = (\delta^1_{\C'} \otimes \id_{\F[U,Q]/(Q)}) (\Phi \otimes \id_{\F[U,Q]/(Q)}) \\
	\Phi_U \tfrac{\delta^1_{\C}}{(Q)} & = \tfrac{\delta^1_{\C'}}{(Q)} \Phi_U,
	\end{align*}
	and $\Phi_Q$ is a chain map because
	\begin{align*}
	(\Phi \delta^1_\C )\otimes \id_{\F[U,Q]/(U)} & = (\delta^1_{\C'} \Phi) \otimes \id_{\F[U,Q]/(U)} \\
	(\Phi \otimes \id_{\F[U,Q]/(U)}) (\delta^1_{\C} \otimes \id_{\F[U,Q]/(U)}) & = (\delta^1_{\C'} \otimes \id_{\F[U,Q]/(U)}) (\Phi \otimes \id_{\F[U,Q]/(U)}) \\
	\Phi_Q \tfrac{\delta^1_{\C}}{(Q)} & = \tfrac{\delta^1_{\C'}}{(Q)} \Phi_Q.
	\end{align*}
	Applying the same method to $\Phi^{-1}$ shows that $\Phi^{-1}_U$ and $\Phi^{-1}_Q$ are chain maps as well. The result follows from the fact that $\Phi^{-1}_{U/Q} \Phi_{U/Q}$ and $\Phi_{U/Q} \Phi^{-1}_{U/Q}$ compose to the identity.
\end{proof}

\subsection{Primitive representatives} \label{subsection_primitive_representatives}

In this section, it will be shown that any decorated immersed representative of a reduced almost $\iota$-complex $\C$ contains a unique immersed curve $\gamma_0(\tCI)$ with no decoration that begins at the $U$-puncture and ends at the $Q$-puncture. This primitive component $\gamma_0(\tCI)$ determines the local equivalence class of $\C$. 

The following lemma dispels the need to distinguish between $\mc{F}(\tCI)$ and $\gamma(\tCI)$ for reduced complexes. 
\begin{lem} \label{reducing_precurve_preserves_local_equivalence_lem}
	For a reduced almost $\iota$-complex $\C$, the $Q^2$-reduction of $\gamma(\tCI)$ is chain isomorphic to $\C$. 
\end{lem}
\begin{proof} 
	It is known that the $Q^2$-reduction of $\mc{F}(\tCI)$ is $\C$. Furthermore, it is known that the simplification to a set of decorated immersed curves does not change the chain isomorphism type of $\mc{F}(\tCI)$. Thus it only needs to be shown that the chain isomorphism $f$ given by the moves taking $\mc{F}(\tCI)$ to $\gamma(\CI)$ preserves the local equivalence class of $\C$. However, if $f$ is a chain isomorphism of $\mc{R}$ complexes, then it descends to give a chain isomorphism of $\mc{R}/(Q^2)$ complexes as well. Thus $\CI$ and the $Q^2$ reduction of $\gamma(\tCI)$ are chain isomorphic as $\mc{R}/(Q^2)$ complexes, and it follows that $\C$ is locally equivalent to the $Q^2$-reduction of $\gamma(\tCI)$. 
\end{proof}

This allows $\C$ to be imagined as $\gamma(\tCI)$ for reduced complexes up to chain isomorphism. It now remains to remove excess immersed curves, and reduce to the primitive case. 
\begin{define} \label{def:primitive_representative}
	The \textit{primitive representative $\gamma_0(\tCI)$} is an immersed curve with no decoration that $Q^2$-reduces to an almost $\iota$-complex which is locally equivalent to $\C$, for any almost $\iota$-complex $\C$.
\end{define}
This concludes the definition of various representatives of $\C$, which is summarized in figure~\ref{fig:table_of_representations_of_c}.
\begin{figure}
	\begin{tabular}{|c|c|c|}
		\hline
		[\ldots] representative of an almost $\iota$-complex $\C$ & Notation & See definition [\ldots] \\
		\hline
		$\mc{R}/(Q^2)$ & $\CI$ & \ref{def:r_mod_q_2_representative_of_c} \\
		$\mc{R}$ & $\tCI$ &  \ref{def:r_representative_of_c} \\
		Precurve & $\mc{F}(\tCI)$ & \ref{def:precurve_representative_of_c} \\
		Multicurve & $\gamma(\tCI)$ & \ref{def:multicurve_representative} \\
		Primitive & $\gamma_0(\tCI)$ &  \ref{def:primitive_representative} \\
		\hline
	\end{tabular}
	\caption{The notation for different representations of an almost $\iota$-complex $\C$.}
	\label{fig:table_of_representations_of_c}
\end{figure}

It will be shown in Theorem~\ref{thm:primitive_rep_of_almost_iota_complex} that any almost $\iota$-complex has a primitive representative which $Q^2$-reduces to the standard representative of $\C$. 

At first, it will only be necessary to work with reduced almost $\iota$-complexes. The following lemma is needed to prove that any reduced almost $\iota$-complex has a primitive representative.

\begin{lem} \label{lem:u_puncture_gives_nontorsion_homology}
	For any reduced almost $\iota$-complex $\C$, there is a unique generator in $\mc{F}(\tCI)$ connecting to the $U$-puncture.
\end{lem}
\begin{proof}
	Since $\C$ is an almost $\iota$-complex, $\rk H_*(\mc{F}(\tCI); \F[U]) = 1$. By Lemma~\ref{lem:rank_of_homology_over_fu_and_fq_from_counting_puncture_connections}, this means there must be a single generator on $s_1(a)$ connecting to the $U$-puncture in $\mc{F}(\tCI)$. 
\end{proof}
\begin{prop} \label{prop:primitive_from_reduced}
	Any reduced almost $\iota$-complex $\C$ has a primitive representative.
\end{prop}
\begin{proof}
	Theorem~\ref{thm:class_decorated_immersed} gives that $C(\mc{F}(\tCI); \mc{R})$ is chain isomorphic to $C(\gamma(\tCI); \mc{R})$. Lemma~\ref{lem:u_puncture_gives_nontorsion_homology} guarantees that $\rk H_*(\mc{F}(\tCI); \F[U]) = 1$, and Lemma~\ref{lem:r_chain_iso_descent} guarantees that 
	$$H_*(\mc{F}(\tCI);\F[U]) \cong H_*(\gamma(\tCI); \F[U]),$$ 
	so Lemma~\ref{lem:rank_of_homology_over_fu_and_fq_from_counting_puncture_connections} gives that there must be a unique generator in $\gamma(\tCI)$ that connects to the $U$-puncture.
	
	Enumerate the components of $\gamma(\tCI)$ by $\{ (\gamma_0, X_0), \cdots, (\gamma_n, X_n) \}$. Assume without loss of generality that $\gamma_0$ is the unique curve connecting the $U$-puncture. If $X_0$ were nontrivial, then after drawing the precurve $\gamma(\tCI)$ using construction~\ref{construction:drawing_decorations}, there would multiple generators on $\gamma(\CI)$ connecting to the $U$-puncture. This would give that $\rk H_*(\gamma(\tCI) ; \F[U]) > 1$ by Lemma~\ref{lem:rank_of_homology_over_fu_and_fq_from_counting_puncture_connections}.
	
	The curves $\{ (\gamma_0, X_0), \cdots, (\gamma_n, X_n) \}$ induce a splitting on the $Q^2$-reduction of $\gamma(\tCI)$ into $\C_0 \oplus \mc{B}_1 \oplus \cdots \oplus \mc{B}_n$, where $\C_0$ is the $Q^2$-reduction of $(\gamma_0, X_0)$ and $\mc{B}_i$ is the $Q^2$-reduction of $(\gamma_i, X_i)$ for $i \neq 0$. By Lemma~\ref{lem:deleting_summands_almost_iota}, this gives that the $Q^2$-reduction of $\gamma(\tCI)$ is locally equivalent to $\C_0$. Since $\C_0$ is the $Q^2$-reduction of $(\gamma_0, X_0)$, this proves that $\gamma_0$ is a primitive representative of $\C$. Therefore any reduced almost $\iota$-complex has a primitive representative.
\end{proof}

\subsection{Proving the classification theorem} \label{subsection_classification}

Now that a primitive representative can be assigned to any reduced almost $\iota$-complex $\C$, the classification of almost $\iota$-complexes is nearly complete. What remains to be done is to classify the possible $\gamma_0(\tCI)$, and then to show that two almost $\iota$-complexes $\C_1$ and $\C_2$ are locally equivalent if and only if $\gamma_0(\tI{\C_1}) \cong \gamma_0(\tI{\C_2})$. 

\begin{lem} \label{lem_primitive_is_standard}
	Given a reduced almost $\iota$-complex $\C$, the $Q^2$-reduction of $\gamma_0(\tCI)$ is a standard complex. 
\end{lem}
\begin{proof}
	The argument to show that $\gamma_0(\tCI)$ represents a standard complex is purely combinatorial. Proposition~\ref{prop:primitive_from_reduced} guarantees that $\gamma_0(\CI)$ is an immersed curve with no decoration. Therefore $\gamma_0(\tCI)$ must start from the $U$-puncture and continue to $N(a)$, as shown in figure~\ref{fig:gamma0_start}.
	
	\begin{figure}[ht!]
		\centering
		\def\svgwidth{0.325\textwidth}
		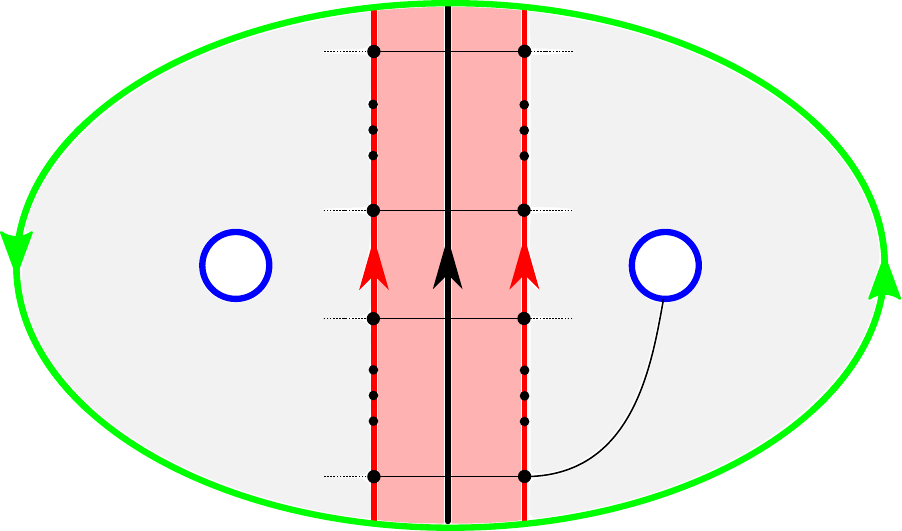
		\caption{Shown above is the connection of the unique tower element with the $U$-puncture. The other segments of $\gamma_0(\CI)$ are yet to be determined.}
		\label{fig:gamma0_start}
	\end{figure}

	This proof will traverse inductively up the spine of $a$, determining each possible segment of $\gamma_0(\tCI)$ until it is shown that the $Q^2$-reduction of $\gamma_0(\tCI)$ is a standard complex. Note that by homotoping $\gamma_0(\tCI)$, it can be assumed that each arc connects a generator on $s_i(a)$ to the generator directly above it. The generators on $s_1(a)$ are labelled by $x_0^1, \cdots, x_n^1$ and the generators on $s_2(a)$ are labelled $x_0^2, \cdots, x_n^2$, from bottom to top. Here it is implicit the $Q^2$ reduction of $\gamma_0(\tCI)$ is generated by $x_1, \cdots, x_n$, that $\mc{G}(x_i^1) = x_i$, and that $\mc{G}(x_i^2) = 0$. To work out explicitly how $\mc{G}$ acts on the morphisms, see equation~\ref{equation:mcG_on_morphisms}. In figure~\ref{fig:gamma0_start}, $x_0^1$ is the only completed generator. Now to continue from $x_0^2$. The segment connecting $x_0^2$ can wrap around the $Q$-puncture at most once by Lemma~\ref{lem:one_q_loop}. Thus there are three possibilities, shown in figure~\ref{fig:gamma0_base_case}. 
	
	\begin{figure}[ht!]
		\begin{subfigure}{.325\textwidth}
			\def\svgwidth{0.95\textwidth}
			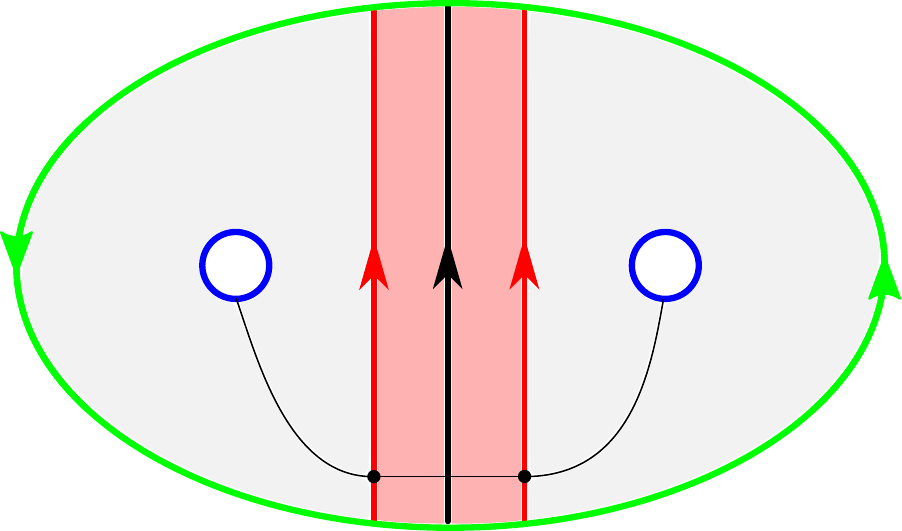
			\label{fig:gamma0_base_case_trivial}
		\end{subfigure}
		\begin{subfigure}{.325\textwidth}
			\def\svgwidth{0.95\textwidth}
			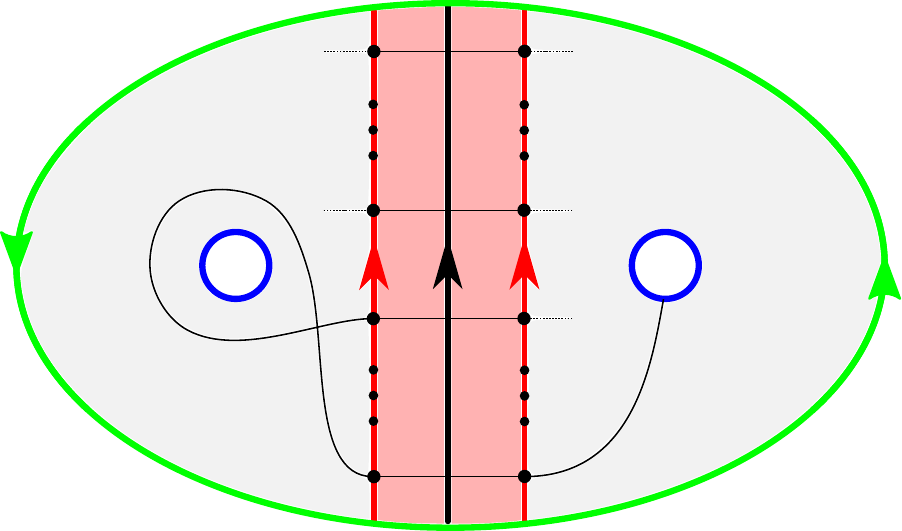
			\label{fig:gamma0_base_case_twist}
		\end{subfigure}
		\begin{subfigure}{.325\textwidth}
			\def\svgwidth{0.95\textwidth}
			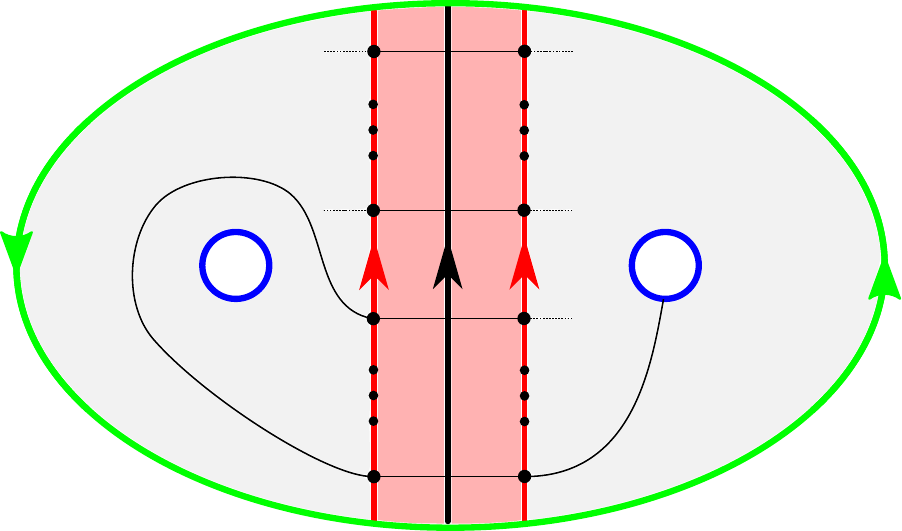
			\label{fig:gamma0_base_case_no_twist}
		\end{subfigure}
		\caption{The three possible ways to complete $x_0^2$.}
		\label{fig:gamma0_base_case}
	\end{figure}

	In the first case, $x_0^2$ is connected to the $Q$-puncture, terminating $\gamma_0(\tCI)$.  In this case, the $Q^2$-reduction of $\gamma_0(\tCI)$ is the trivial standard complex $\C(0)$. In the second case, $x_0^2$ is connected to $x_0^1$ by a segment that sets $\omega x_0 = x_1$ in the $Q^2$-reduction of $\gamma_0(\tCI)$. Setting $\omega x_0 = x_1$ corresponds to setting $a_1 = -$ in the standard complex parametrization. In the last case, $x_0^2$ is connected to $x_1^2$ by a segment that wraps once around the $Q$-puncture in a way that sets $\omega x_1 = x_0$ in the $Q^2$-reduction of $\gamma_0(\tCI)$. This corresponds to setting $a_1 = +$ in the standard complex parametrization.
	
	Suppose all generators $x_0^1, \cdots, x_{i-1}^1$ and $x_0^2, \cdots, x_{i}^2$ have been completed, and the segment connecting $x_i^1$ is yet to be determined. This can only happen if $i$ is odd. Then there are two possible ways to connect $x_i^1$, shown in figure~\ref{fig:gamma0differentialinduction}. Notice that neither possiblility terminates $\gamma_0(\tCI)$, since Lemma~\ref{lem:rank_of_homology_over_fu_and_fq_from_counting_puncture_connections} would force $\rk H_*(\gamma_0(\tCI);\F[U]) = 2$ if the curve went to the $U$-puncture again. However, this is impossible by Lemma~\ref{lem:u_puncture_gives_nontorsion_homology}. In the first case, $x_{i}^1$ is connected to $x_{i+1}^1$ by a segment that travels $n$ times around the $U$-puncture in a way that sets $\p x_{i} = U^n x_{i+1}$ in the $Q^2$-reduction of $\gamma_0(\tCI)$. This corresponds to setting $b_{i+1} = -n$ in the standard complex parametrization. In the second case, $x_{i}^1$ is connected to $x_{i+1}^1$ by a segment that travels $n$ times around the $U$-puncture in a way that sets $\p x_{i+1} = U^n x_{i}$ in the $Q^2$-reduction of $\gamma_0(\tCI)$. This corresponds to setting $b_{i+1} = n$ in the standard complex parametrization.
	
	\begin{figure}[ht!]
		\begin{subfigure}{.325\textwidth}
			\def\svgwidth{0.95\textwidth}
\begingroup%
  \makeatletter%
  \providecommand\color[2][]{%
    \errmessage{(Inkscape) Color is used for the text in Inkscape, but the package 'color.sty' is not loaded}%
    \renewcommand\color[2][]{}%
  }%
  \providecommand\transparent[1]{%
    \errmessage{(Inkscape) Transparency is used (non-zero) for the text in Inkscape, but the package 'transparent.sty' is not loaded}%
    \renewcommand\transparent[1]{}%
  }%
  \providecommand\rotatebox[2]{#2}%
  \newcommand*\fsize{\dimexpr\f@size pt\relax}%
  \newcommand*\lineheight[1]{\fontsize{\fsize}{#1\fsize}\selectfont}%
  \ifx\svgwidth\undefined%
    \setlength{\unitlength}{432.53171145bp}%
    \ifx\svgscale\undefined%
      \relax%
    \else%
      \setlength{\unitlength}{\unitlength * \real{\svgscale}}%
    \fi%
  \else%
    \setlength{\unitlength}{\svgwidth}%
  \fi%
  \global\let\svgwidth\undefined%
  \global\let\svgscale\undefined%
  \makeatother%
  \begin{picture}(1,0.58915273)%
    \lineheight{1}%
    \setlength\tabcolsep{0pt}%
    \put(0,0){\includegraphics[width=\unitlength,page=1]{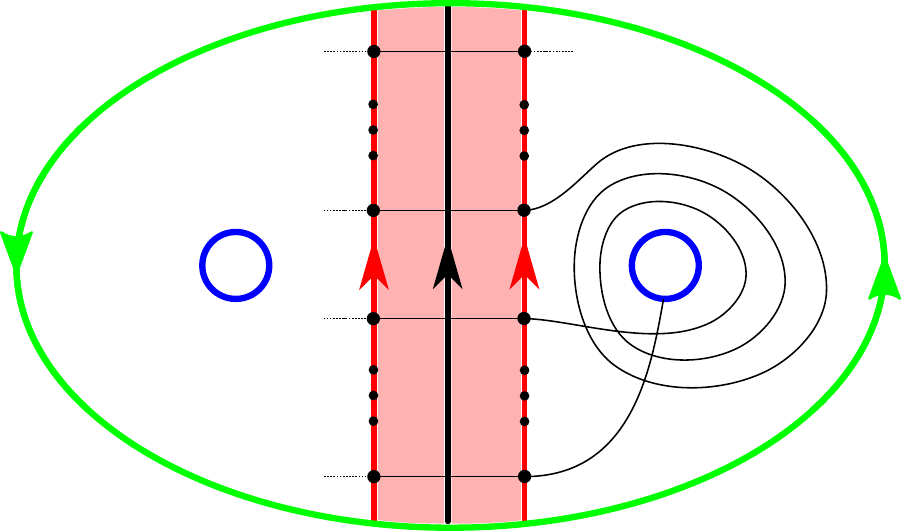}}%
    \put(0.91939436,0.32776969){\makebox(0,0)[lt]{\lineheight{1.25}\smash{\begin{tabular}[t]{l}$n$\end{tabular}}}}%
    \put(0,0){\includegraphics[width=\unitlength,page=2]{gamma_0_induction_U_+.pdf}}%
  \end{picture}%
\endgroup%

			\label{fig:gamma0_induction_U_+}
		\end{subfigure}
		\begin{subfigure}{.325\textwidth}
			\def\svgwidth{0.95\textwidth}
\begingroup%
  \makeatletter%
  \providecommand\color[2][]{%
    \errmessage{(Inkscape) Color is used for the text in Inkscape, but the package 'color.sty' is not loaded}%
    \renewcommand\color[2][]{}%
  }%
  \providecommand\transparent[1]{%
    \errmessage{(Inkscape) Transparency is used (non-zero) for the text in Inkscape, but the package 'transparent.sty' is not loaded}%
    \renewcommand\transparent[1]{}%
  }%
  \providecommand\rotatebox[2]{#2}%
  \newcommand*\fsize{\dimexpr\f@size pt\relax}%
  \newcommand*\lineheight[1]{\fontsize{\fsize}{#1\fsize}\selectfont}%
  \ifx\svgwidth\undefined%
    \setlength{\unitlength}{432.53171145bp}%
    \ifx\svgscale\undefined%
      \relax%
    \else%
      \setlength{\unitlength}{\unitlength * \real{\svgscale}}%
    \fi%
  \else%
    \setlength{\unitlength}{\svgwidth}%
  \fi%
  \global\let\svgwidth\undefined%
  \global\let\svgscale\undefined%
  \makeatother%
  \begin{picture}(1,0.58915273)%
    \lineheight{1}%
    \setlength\tabcolsep{0pt}%
    \put(0,0){\includegraphics[width=\unitlength,page=1]{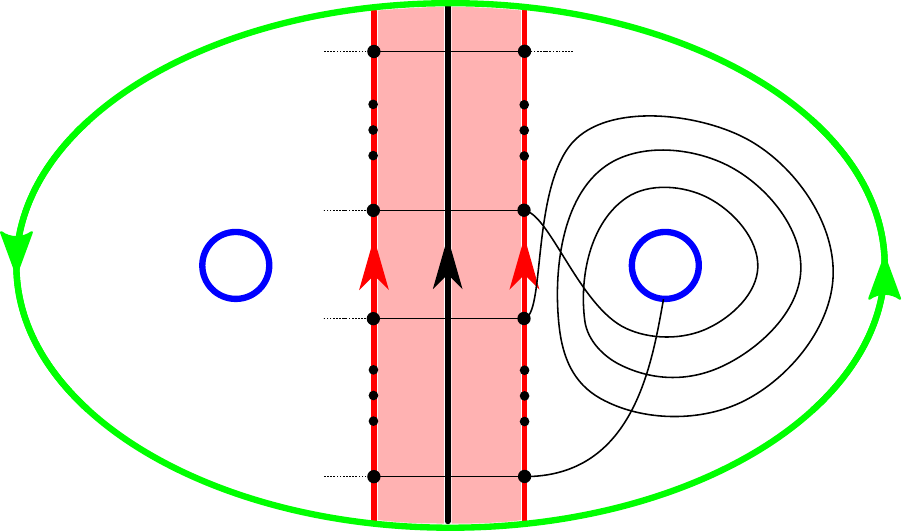}}%
    \put(0.93039329,0.32676109){\makebox(0,0)[lt]{\lineheight{1.25}\smash{\begin{tabular}[t]{l}$n$\end{tabular}}}}%
    \put(0,0){\includegraphics[width=\unitlength,page=2]{gamma_0_induction_U_-.pdf}}%
  \end{picture}%
\endgroup%

			\label{fig:gamma0_induction_U_-}
		\end{subfigure}
		\caption{The two possible ways to complete $x_{i}^1$ in $\gamma_0(\CI)$.}
		\label{fig:gamma0differentialinduction}
	\end{figure}
	
	Now suppose that all the generators $x_0^1, \cdots, x_{i}^1$ and $x_0^2, \cdots, x_{i-1}^2$ have been completed, and the segment connecting $x_{i}^2$ is not yet determined. According to the inductive process, this will only happen for even $i$.  Then Lemma~\ref{lem:one_q_loop} again requires that any segment from $x_{i}^2$ wraps at most one time about the $Q$-puncture. This gives three cases, shown in figure \ref{fig:gamma0involutioninduction}. The first possibility sends $x_{i}^2$ to the $Q$-puncture, and terminates $\gamma_0(\tCI)$. In this case, $\gamma_0(\tCI)$ $Q^2$-reduces to the standard complex parametrized by the previous steps. In the second case, $x_{i}^2$ is connected to $x_{i+1}^2$ by a segment wrapping around the $Q$-puncture in a way that sets $\omega x_{i} = x_{i+1}$ in the $Q^2$-reduction of $\gamma_0(\tCI)$. This corresponds to setting $a_{i+1} = -$ in the standard complex parametrization. In the third case, $x_{i}^2$ is connected to $x_{i+1}^2$ by a segment wrapping around the $Q$-puncture in a way that sets $\omega x_{i+1} = x_{i}$ in the $Q^2$-reduction of $\gamma_0(\tCI)$. This corresponds to setting $a_{i+1} = +$ in the standard complex parametrization.
	
	\begin{figure}[ht!]
		\begin{subfigure}{.325\textwidth}
			\def\svgwidth{0.95\textwidth}
			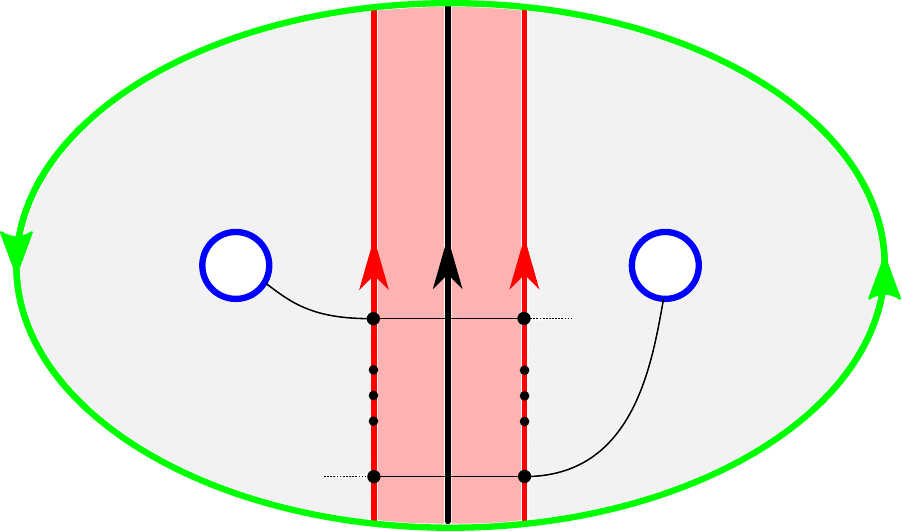
			\label{fig:gamma0_induction_Q_end}
		\end{subfigure}
		\begin{subfigure}{.325\textwidth}
			\def\svgwidth{0.95\textwidth}
			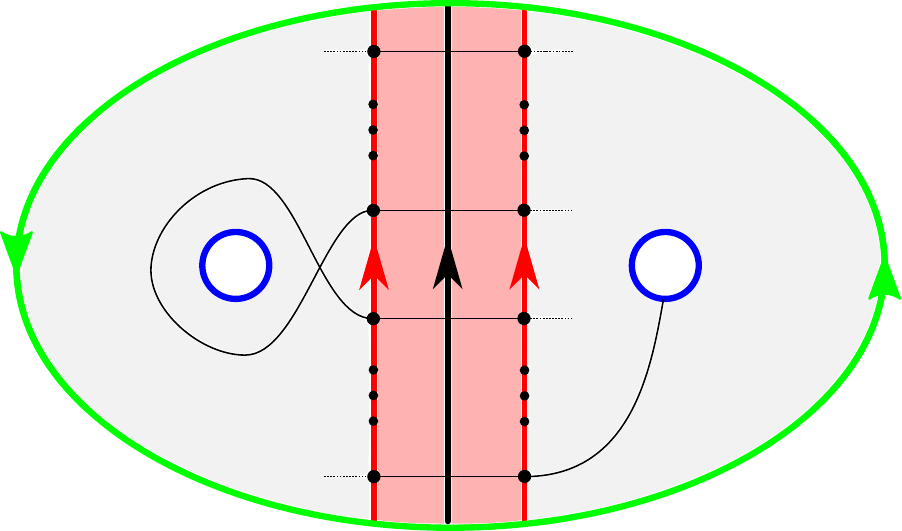
			\label{fig:gamma0_induction_Q_-}
		\end{subfigure}
		\begin{subfigure}{.325\textwidth}
			\def\svgwidth{0.95\textwidth}
			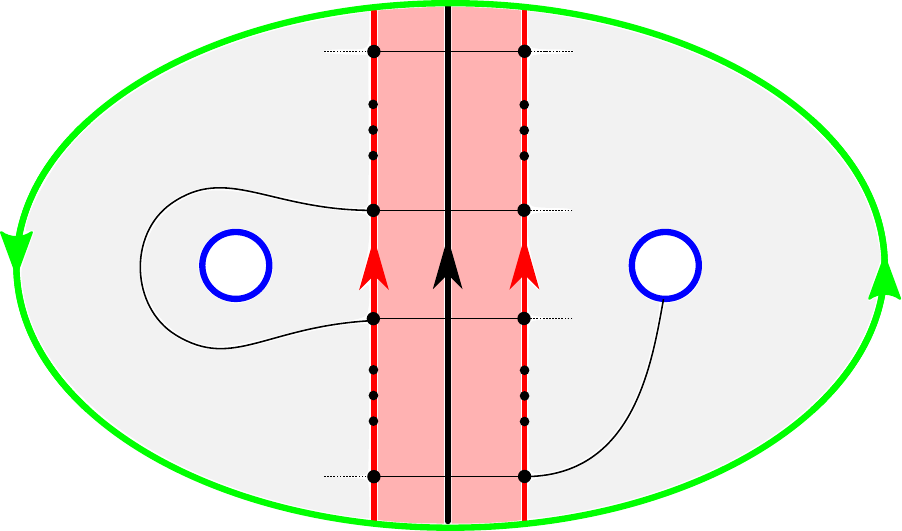
			\label{fig:gamma0_induction_Q_+}
		\end{subfigure}
		\caption{The three possible ways to complete $x_{i-1}^2$ and $x_{i}^2$ for odd $i$ in $\gamma_0(\CI)$.}
		\label{fig:gamma0involutioninduction}
	\end{figure}

	Inductively, this proves that the $Q^2$ reduction of $\gamma_0(\tCI)$ is a standard complex, as eventually this process must terminate by ending at the $Q$-puncture.
\end{proof}

\begin{thm}\label{thm:primitive_rep_of_almost_iota_complex}
	For any almost $\iota$ complex $\C$, there is a unique primitive representative $\gamma_0(\tCI)$, and $\gamma_0(\tCI)$ $Q^2$-reduces to the standard complex representative of $\C$.
\end{thm}
\begin{proof}
	Suppose $\C$ is locally equivalent to two different reduced standard complexes $\C_1$ and $\C_2$. Then $\C_1$ is locally equivalent to the $Q^2$-reduction of $\gamma_0(\tI{\C_1})$ and $\C_2$ is locally equivalent to the $Q^2$-reduction of $\gamma_0(\tI{\C_2})$. This gives is a local equivalence between the $Q^2$-reductions of $\gamma_0(\tI{\C_1})$ and $\gamma_0(\tI{\C_2})$, which are both standard complexes. By Theorem 4.6 of \cite{dai2018infiniterank}, this proves that $\gamma_0(\tI{\C_1}) = \gamma_0(\tI{\C_2})$.
\end{proof}

\begin{proof}[Proof of Theorem~\ref{thm:hi_classification}]
	Every almost $\iota$-complex $\C$ is locally equivalent to the $Q^2$-reduction of $\gamma_0(\tCI)$, which is a standard complex.
\end{proof}

\section{Homomorphisms from $\hI$ to $\Z$} \label{section:homomorphisms}

This section will describe all known homomorphisms from $\hI$ to $\Z$ using precurves. First, it will be proven that $P$ and $P_\omega$ are homomorphisms, and then it will be shown that the maps $\phi_n : \hI \to \Z$ are homomorphisms as well. The maps $P$ and $P_\omega$ from $\hI$ to $\Z$ are defined as follows:
\begin{define} \label{def:pivotal_map}
	Let $\C$ be the standard complex representative of an element of $\hI$, and let $\C$ be generated by $T_0, \cdots, T_{2n}$. Then $P(\C)$ is the grading of the final generator $T_{2n}$.
\end{define}
\begin{define} \label{def:p_omega}
	Let $\C = \C(a_i, b_i)$ be the standard complex representative of an element of $\hI$, then 
	$$P_\omega(\C) = \# \{ a_i = + \} - \# \{ a_i = - \}.$$
\end{define}

Recall the functor $\mc{G} : \Cx \Mat_i \bar{\A} \to \Cx \Mat \A$. 
The heart of the proofs in this section come down to studying the following elements of a standard complex:

\begin{define}
	For a standard almost $\iota$-complex $\C$, the \textit{initial generator} $x_i(\C)$ is the unique element in $\mc{G}(\gamma_0(\tCI))$ which is in $\ker \p$ and for which there is no power of $U$ in $\im \p$. 
\end{define}
It is best to think of $x_i(\C)$ as the image of the generator connecting to the $U$-puncture on $\gamma_0(\tCI)$
\begin{define}
	For a standard almost $\iota$-complex $\C$, the \textit{final generator} $x_f(\C)$ is the unique generator in $\mc{G}(\gamma_0(\tCI))$ which is in $\ker Q \omega$ and for which there is no power of $Q$ in $\im Q \omega$. 
\end{define}
It is best to think of $x_f(\C)$ as roughly corresponding to the generator connecting to the $Q$ puncture on $\gamma_0(\tCI)$. 
In standard complexes $\C$ generated by $T_0, \cdots, T_{2n}$, $x_i(\C)$ is $T_0$ and $x_{f}(\C)$ is $T_{2n}$. 

The fact that $P$ and $P_\omega$ are homomorphisms will be shown using a tensor product formula proven in subsection~\ref{subsection:a_tensor_product_formula} and special bigrading constructed in subsection~\ref{subsection:counting_bigrading_of_final_generator}. 

Precurves are particularly well adapted to showing that the shift maps $\sh_n$ from \cite{dai2018infiniterank} are endomorphisms of $\hI$, and this will be used in subsection~\ref{subsection:shift_maps_and_phi_n_with_precurves} to prove that the maps $\phi_n$ from \cite{dai2018infiniterank} are homomorphisms.

\subsection{A tensor product formula for almost $\iota$-complexes} \label{subsection:a_tensor_product_formula}
Let $\C = (C, \p, \biota)$ be an almost $\iota$-complex generated by $\{x_1, \cdots, x_n\}$ with $\F$ span $X$. 
The complex $\tCI$ admits a natural $\mc{R}$ bimodule structure $\cdot$. Let $\p_i \in \mc{R}$. The left $\mc{R}$-module structure $\cdot: \mc{R} \times \tCI \to \tCI$ is defined by 
$$U \cdot (\sum p_i x_i) = \sum U p_i x_i, \qquad Q \cdot (\sum p_i x_i) = \sum Q p_i x_i.$$
The left $\mc{R}$-module structure $\cdot : \tCI \times \R \to \tCI$ is defined by
$$(\sum p_i x_i) \cdot U = \sum U p_i x_i, \qquad (\sum p_i x_i) \cdot Q = \sum Q p_i x_i.$$
To prove the tensor product formula, it will be necessary to define a new $\mc{R}$ bimodule structure $\tCI'$.
\begin{define}
	The elements of $\tCI'$ are $\mc{R} \otimes_{\F} X$. The left $\R$ module structure of $\tCI'$ is defined by
	$$U \cdot (\sum p_i x_i) = \sum U p_i x_i, \qquad Q \cdot (\sum p_i x_i) = \sum Q p_i x_i$$
	and the right $\R$ module structure of $\tCI'$ is defined by
	$$ (\sum p_i x_i) \cdot U =  \sum U p_i x_i, \qquad  (\sum p_i x_i) \cdot Q = \sum Q p_i \biota x_i.$$
	If $\C$ is reduced, then $\tCI'$ is a chain complex with differential
	$$\delta^1_{\tI{\C}'} = \p + Q \omega.$$
\end{define}
\begin{rem}
Recall from Definition 3.15 of \cite{dai2018infiniterank} that $\hI$ is a group where the elements of $\hI$ are the equivalence classes of almost $\iota$-complexes and the operation is the tensor product of those classes. Here we denote that operation by $*$ to avoid conflating the group operation in $\hI$ with the tensor product symbol.
\end{rem}
Note that if $\C_1$ and $\C_2$ are reduced complexes, then $\C_1 *\C_2$ is as well. The lift of $\C_1 * \C_2$ to an $\mc{R}$ bimodule is denoted $\tI{(\C_1 * \C_2)}$.
\begin{thm}[A tensor product formula for almost $\iota$-complexes] \label{thm:tensor_product_formula_for_almost_iota_complexes}
	If $\C_1$ and $\C_2$ are reduced almost $\iota$-complexes, then
	$$\tI{(\C_1 * \C_2)} \cong \tI{\C_1}' \otimes \tI{\C_2}$$
	as left $\mc{R}$ modules and $\mc{R}$ complexes.
\end{thm}
The proof of this theorem is the focus of this subsection.
\begin{rem}
	It does not seem that $\tI{(\C_1 * \C_2)}$ and $\tI{\C_1}' \otimes \tI{\C_2}$ are naturally isomorphic as right $\mc{R}$ modules.
\end{rem}

To prove a tensor product formula for almost $\iota$-complexes, it will need to be shown that $\tI{\C_1}' \otimes \tI{\C_2}$ is free as a left $\mc{R}$ module. The following proposition will be proven shortly:
\begin{prop} \label{prop:freeness_of_tensor_product}
	If $\C_1$ and $\C_2$ are reduced almost $\iota$-complexes, then
	$\tI{\C_1}'\otimes \tI{\C_2} $
	is a free left $\mc{R}$ module. 
\end{prop}
When an almost $\iota$-complex $\C$ is reduced, $\omega^2 \equiv 0$ mod $U$. Before, this gave a lift of $\C$ to an $\mc{R}$ complex $\tCI$. Now the fact that $\omega^2 \equiv 0$ mod $U$ for reduced complexes will be central in showing that $\tI{\C_1}'\otimes \tI{\C_2} $ is free.

\begin{lem} \label{lem:q_biota_is_an_isomorphism}
	If $\C$ is reduced, then $Q \biota$ is an $\F$ vector space isomorphism between $X$ and $QX$.
\end{lem}
\begin{proof}[Proof by contradiction]
	Observe that $X$, $QX$, and $Q^2 X$ are $\F$ vector spaces with respect to the bases $\{Q^i x_1, \cdots, Q^i x_n\}$ for $i \in \{0,1,2 \}$. Since
	$Q \biota: X \to QX$, $Q \biota : QX \to QX^2$, and $Q^2 \biota^2 : X \to Q^2 X$ are all $\F$ linear maps, 
	\begin{center}
		\begin{tikzcd}
		X \arrow[r, "Q \biota"] \arrow[bend left=30]{rr}{Q^2 \biota^2} & Q X \arrow[r, "Q \biota"] & Q^2 X 
		\end{tikzcd}
	\end{center}
	is a commutative diagram of vector spaces.
	Since $\C$ is reduced, $\biota^2 \equiv 1 \mod U$ so $Q \biota^2 = Q$. It then follows that $(Q \biota)^2 = Q^2 \biota^2 = Q^2$, so $Q^2 \biota^2$ is an isomorphism between $X$ and $Q^2 X$. If $Q \biota$ were not an isomorphism, then the composite map $Q^2 \biota^2$ could not be an isomorphism either. 
\end{proof}

It can now be proven that $\tCI'$ is free as a right $\mc{R}$ module using Lemma~\ref{lem:q_biota_is_an_isomorphism}.
Note that $\tCI'$ is already clearly free as a left $\mc{R}$ module.

\begin{lem} \label{lem:ci_prime_is_free_right_r_module}
	If $\C$ is a reduced almost $\iota$-complex, then $\tCI'$ is a free right $\R$ module.
\end{lem}
\begin{proof}
	Denote the generators of $\C$ by $\{x_1, \cdots, x_n \}$. Suppose there were some polynomials $p_i \in \mc{R}$ such that
	$$\sum x_i \cdot p_i = 0$$
	in $\tCI'$. Write $p_i$ as $c_i + U p_i^U + Q p_i^Q$ where $c_i \in \F, p_i^U \in \F[U]$ and $p_i^Q \in \F[Q]$. If $\sum x_i \cdot p_i = 0$, then
	$$ \sum x_i \cdot c_i + \sum x_i \cdot U p_i^U + \sum x_i \cdot Q p_i^Q = 0.$$
	Constant coeffiecients only appear in $\sum x_i \cdot c_i$, which is zero if and only if $c_i = 0$ for all $i$. The only $U$ coefficients appear in $ \sum x_i \cdot U p_i^U$. However
	$$\sum x_i \cdot U p_i^U = \sum U p_i^U x_i$$
	is zero if and only if $p_i^U = 0$ for all $i$. The only $Q$ coeffiecients appear in $ \sum x_i \cdot Q p_i^Q$. Let $k$ be the maximal degree of any polynomial $p_i^Q$. If $p_i^Q$ is rewritten as 
	$$ p_i^Q = a_i^0 + \cdots + a_i^{j} Q^j + \cdots a_i^{k} Q^{k} = \sum_{j=0}^k a_i^j Q^j$$ 
	where $a_i^{j} = 0$ for $j > \deg p_i^Q$, then
	\begin{align*}
	\sum_i x_i \cdot Q p_i^Q & = \sum_{j=0}^k  \sum_i x_i \cdot a_i^j Q^j  = \sum_{j=0}^k  Q^j \sum_i a_i^j Q \biota^{j+1} x_i.
	\end{align*}
	Since $\C$ is reduced, $\biota^2 \equiv 1 \mod U$ and $Q \biota^2 = Q$. 
	Therefore $Q \biota^j = Q$ if $j$ is even and $Q \biota^j = Q \biota$ if $j$ is odd. Consider each power of $Q$ in the sum separately. If $j$ is odd, then 
	$$Q^{j} \sum a_i^{j} Q \biota^{j+1} x_i = Q^j \sum a_i^j Q  x_i = Q^{j+1} \sum a_i^j x_i$$
	so 
	$$Q^{j} \sum a_i^{j} Q \biota^{j+1} x_i = 0$$
	if and only if $a_i^j = 0$ for all $i$. If $j$ is even, then
	$$Q^j \sum a_i^j  Q \biota^{j+1} x_i = Q^j \sum a_i^j Q \biota x_i = \sum Q^{j+1} a_i^j \biota x_i$$
	By Lemma~\ref{lem:q_biota_is_an_isomorphism}, $Q \biota x_i$ is a basis for $QX$, so 
	$$Q^j \sum a_i^j  Q \biota^{j+1} x_i = 0$$
	if and only if $a_i^j = 0$ for all $i$. Since $a_i^j = 0$ for all $i$ and $j$, it must be that $p_i^Q$ is zero as well. This proves that 
	$$\sum x_i \cdot p_i = 0 \qquad \iff \qquad  p_i = 0\text{ for all $i$},$$
	so $\tCI'$ is a free right $\mc{R}$ module.
\end{proof}

It is now possible to prove Proposition~\ref{prop:freeness_of_tensor_product}.

\begin{proof}[Proof of Proposition~\ref{prop:freeness_of_tensor_product}]
	Note that $\tI{\C_1}'$ is free as a right $\mc{R}$ module by Lemma~\ref{lem:ci_prime_is_free_right_r_module} and that $\tI{\C_2}$ is free as a left and right $\mc{R}$ module. It then follows that $\tI{\C_1}' \otimes \tI{\C_2}$ is free as a right $\mc{R}$ module.
	
	Let $\{x_1,\cdots , x_m\}$ be a basis for $X_1$ and $\{y_1, \cdots, y_n\}$ be a basis for $X_2$. 
	Suppose there were $p_{ij} \in \mc{R}$ such that
	$$\sum p_{ij} \cdot (x_i \otimes y_j) = 0$$
	in $\tI{\C_1}' \otimes \tI{\C_2}$. Let $p_{ij} = p_{ij}^U + Q p_{ij}^Q$ where $p_{ij}^U \in \F[U]$ and $p_{ij}^Q \in \F[Q]$. Note that the constant term is considered to be a term in $p_{ij}^U$. Then 
	$$\sum p_{ij} \cdot (x_i \otimes y_j) = \sum ((p_{ij}^U + Q p_{ij}^Q) \cdot x_i) \otimes y_j.$$
	Recall that $x_i \cdot Q = Q \biota_1 x_i = Q \cdot (\biota_1 \mod U) x_i$ in $\tI{\C_1}'$. Let $d$ be the maximal degree of any polynomial $p_{ij}^Q$, and let $p_{ij}^Q = \sum_{k=0}^d a_{ij}^k Q^k$ where $a_{ij}^k = 0$ if $k > \deg p_{ij}^Q$. Then 
	\begin{align*}
		\sum_{i,j} ((p_{ij}^U & + Q p_{ij}^Q) \cdot x_i) \otimes y_j  = \sum_{ij} (x_i \otimes y_j) \cdot p_{ij}^U + \sum_{i,j} Q p_{ij}^Q \cdot x_i \otimes y_j \\ 
		& = \sum_{ij} (x_i \otimes y_j) \cdot p_{ij}^U + \sum_{k=0}^d \sum_{i,j} ((\biota_1 \mod U)^{k+1} x_i \otimes y_j) \cdot (a_{ij}^k Q^{k+1}).
	\end{align*}
	Considering each constant term or power of $U$, it follows that $p_{ij}^U = 0$ since $\tI{\C_1}' \otimes \tI{\C_2}$ is free as a right $\mc{R}$ module. Note also that $(\biota_1 \mod U)$ is an isomorphism of $\tCI'$ as a right $\mc{R}$ module by Lemma~\ref{lem:q_biota_is_an_isomorphism}. Thus considering each power of $Q$ separately, it follows that $a_{ij}^k = 0$ for all $i,j,k$ since $(\biota_1 \mod U)^{k+1} x_i \otimes y_j$ is a basis for $\tI{\C_1}' \otimes \tI{\C_2}$ as a free right $\mc{R}$ module. This proves that $p_{ij} = 0$ for all $i$ and $j$, so it follows that $\tI{\C_1}' \otimes \tI{\C_2}$ is free as a left $\mc{R}$ module.
\end{proof}

With freeness, it is possible to prove the tensor product formula.
Define the maps $F: \tI{(\C_1 * \C_2)} \to \tI{\C_1}' \otimes \tI{\C_2}$ and $G: \tI{\C_1}' \otimes \tI{\C_2} \to \tI{(\C_1 * \C_2)}$ by
$$F(\sum p_{ij} \cdot (x_i \otimes y_j) )= \sum p_{ij} \cdot (x_i \otimes y_j)$$
and 
$$G(\sum p_{ij} \cdot (x_i \otimes y_j)) = \sum p_{ij} \cdot (x_i \otimes y_j).$$
The maps $F$ and $G$ are well defined since $\tI{\C_1}' \otimes \tI{\C_2}$ is free as a left $\mc{R}$ module by Proposition~\ref{prop:freeness_of_tensor_product}. It is easy to verify they are morphisms of left $\mc{R}$ modules as well.
\begin{lem}
	$F$ and $G$ are homomorphisms of left $\mc{R}$ modules.
\end{lem}
\begin{proof} 
	Note that
	$$F(r \cdot \sum p_{ij} \cdot (x_i \otimes y_j)) = \sum  (rp_{ij}) \cdot (x_i \otimes y_j) = r \cdot F(\sum p_{ij}\cdot (x_i \otimes y_j))$$
	and
	$$G(r \cdot \sum p_{ij} \cdot (x_i \otimes y_j)) = \sum  (rp_{ij}) \cdot (x_i \otimes y_j) = r \cdot G(\sum p_{ij}\cdot (x_i \otimes y_j))$$
	for any $r \in \R$. 
\end{proof}
In other words, $F$ and $G$ are inverse isomorphisms of left $\mc{R}$ modules. Once it is verified that $F$ and $G$ are chain maps, the tensor product formula is proven.

\begin{proof}[Proof of Theorem~\ref{thm:tensor_product_formula_for_almost_iota_complexes}]
	Note that 
	\begin{align*}
	\delta^1_{\tI{\C_1}' \otimes \tI{\C_2}} & = \delta^1_{\tI{\C_1}'} \otimes 1 + 1 \otimes \delta^1_{\tI{\C_2}} \\
	& = (\p_1 + Q (1 + \biota_1)) \otimes 1 + 1 \otimes (\p_2 + Q (1 + \biota_2)) \\
	& = \p_1 \otimes 1 + 1 \otimes \p_2 + Q \otimes 1 + Q \biota_1 \otimes 1 + 1 \otimes Q + 1 \otimes Q \biota_2 \\
	& = \p_1 \otimes 1 + 1 \otimes \p_2 + Q \otimes 1 + Q \biota_1 \otimes \biota_2 
	\end{align*}
	since $Q \biota_1 \otimes 1 = 1 \otimes Q$ in $\tI{\C_1}' \otimes \tI{\C_2}$. It then follows that
	\begin{align*}
	F \delta^1_{\tI{(\C_1 * \C_2)}}(\sum p_{ij} \cdot (x_i \otimes y_j)) & = F(\sum (p_{ij} \p_1 x_i \otimes y_j + p_{ij} x_i \otimes y_j + Q x_i \otimes y_j + Q \biota_1 x_i \otimes \biota_2 y_j)) \\
	& = \sum (p_{ij} \p_1 x_i \otimes y_j + p_{ij} x_i \otimes y_j + Q x_i \otimes y_j + Q \biota_1 x_i \otimes \biota_2 y_j) \\
	& = (\p_1 \otimes 1 + 1 \otimes \p_2 + Q \otimes 1 + Q \biota_1 \otimes \biota_2) \sum p_{ij} x_i \otimes y_j \\
	& = \delta^1_{\tI{\C_1}' \otimes \tI{\C_2}} F(\sum p_{ij} \cdot(x_i \otimes y_j))
	\end{align*}
	and that $F$ is a chain map. The explicit computation can also be done for $G$, but it must be a chain map since $G = F^{-1}$ is a left $\mc{R}$-module isomorphism.
\end{proof}

\subsection{Counting a bigrading of the final generator} \label{subsection:counting_bigrading_of_final_generator}

With the tensor product formula in place, it is now possible to prove that $P$ and $P_\omega$ are homomorphism once a special bigrading is constructed.

\begin{define}
	For a standard complex $\C$, the \textit{na\"ive bigrading} $\gr = (\gr_U, \gr_Q)$ on $\tCI$ is the bigrading specified by $\gr(x_i(\C)) = (0,0)$, $\gr U = (-2,0)$, $\gr Q = (0,-2)$, $\gr \p = (-1,0)$ and $\gr \omega = (0,1)$. 
\end{define}
In the na\"ive bigrading, $\gr_U$ is the \textit{$U$-grading} and $\gr_Q$ is the \textit{$Q$-grading}. In practice, the $U$-grading is the normal grading of a standard complex, and the $Q$-grading of a generator counts with sign how many red arrows have been crossed to reach that generator in the standard complex. The maps $P$ and $P_\omega$ can be understood in terms of this na\"ive bigrading.
\begin{prop}
	For a standard complex $\C$, $P(\C) = \gr_U(x_f(\C))$ and $P_{\omega}(\C) = \gr_Q(x_f(\C))$. 
\end{prop}
\begin{proof}
	The $U$-grading is the usual grading, and $x_f(\C)$ is $T_{2n}$ in standard complexes $\C$ generated by $T_0, \cdots, T_{2n}$. Thus by Definition~\ref{def:pivotal_map}, $P(\C) = \gr_U(x_f(\C))$. Also note that each red arrow in a standard complex corresponds to an action of $Q \omega$ in $\tCI$, which decreases the $Q$-grading by one in the direction which the arrow travels. Summing along all arrows from the initial to the final generator gives that $\gr_Q(x_f(\C)) = P_\omega(\C)$ by Definition~\ref{def:p_omega}.
\end{proof}
\begin{example}
	Consider the standard complex $\C(+,1,-,-2)$. $\tCI$ is given below:
	\begin{center}
		\begin{tikzcd}
		\bullet_{x_i(\C) = T_0} & \arrow[l,red] \bullet_{T_1} & \arrow[l,"1"] \bullet_{T_2} \arrow[r,red] & \bullet_{T_3} \arrow[r,"2"] & \bullet_{x_f(\C) = T_4}
		\end{tikzcd}.
	\end{center}
	Here $\gr(x_i(\C) = T_0) = (0,0)$, $\gr(T_1) = (0,-1)$, $\gr(T_2) = (-1,-1)$, $\gr(T_3) = (-1,0)$, $\gr(T_4) = (2,0)$. Thus $P(\C) = 2$ and $P_\omega(\C) = 0$.
\end{example}
To show that $P$ and $P_\omega$ are homomorphisms it will be necessary to somehow extend the na\"ive bigrading to the product of two standard complexes. This is nontrivial.	
Suppose the bigrading were extended to the product $\tI{(\C_1 * \C_2)}$ by setting 
$$\gr_\circ(x \otimes y) = \gr_\circ(x) + \gr_\circ(y)$$
for $\circ\in \{U, Q\}$. Since $\omega_{\C_1 * \C_2}$ does not follow the Leibniz rule, $\delta^1$ does not necessarily have a homogeneous degree of $-1$. This presents a slight problem, for which there is fortunately a work-around. It turns out that a compatible bigrading can be given to $\tI{\C_1}' \otimes \tI{\C_2}$ considered as a right $\mc{R}$ module. 
\begin{define}
	Given two standard complexes $\C_1$ and $\C_2$, the \textit{product bigrading} $\gr = (\gr_U, \gr_Q)$ is defined on $\tI{\C_1}' \otimes \tI{\C_2}$ by setting
	$$\gr_\circ (\sum x_i \otimes p_{ij} y_j) = \gr_\circ x_i + \gr_\circ p_{ij} + \gr_\circ y_j $$
	for $\circ \in \{ U, Q\}$, where $\gr_{\circ} x_i, \gr_{\circ} y_j,$ and $\gr_{\circ}p_{ij}$ are defined in terms of the na\"ive bigradings on $\C_1$ and $\C_2$. 
\end{define}
\begin{rem}
	It is possible that 
	$$\gr_{\circ} (\sum x_i \otimes p_{ij} y_j) \neq \gr_\circ (\sum p_{ij} x_i \otimes y_j)$$
	since in general it is not true that $\sum x_i \otimes p_{ij} y_j = \sum p_{ij} x_i \otimes y_j$ in $\tI{\C_1}'\otimes \tI{\C_2}$.
\end{rem}
\begin{lem}
	Given any two standard complexes $\C_1$ and $\C_2$, the differential on $\tI{\C_1}' \otimes \tI{\C_2}$ has homogeneous degree $-1$ with respect to the product bigrading.
\end{lem}
\begin{proof} Recall that 
	$$\delta^1_{\tI{\C_1}' \otimes \tI{\C_2}} = \delta^1_{\tI{\C_1}'} \otimes 1 + 1 \otimes \delta^1_{\tI{\C_2}} = \p_1 \otimes 1 + 1 \otimes \p_2 + Q \omega_1 \otimes 1 + 1 \otimes Q \omega_2. $$
	In $\tI{\C_1}'$, 
	$$\omega_1 \cdot Q = Q \biota_1 (1 + \biota_1) = Q(\biota_1 + 1) = Q \omega_1$$
	since $\C_1$ is reduced and $Q \biota_1^2 = Q$ in $\tI{\C_1}'$. Then it follows that $Q \omega_1 \otimes 1 = \omega_1 \cdot Q \otimes 1 = \omega_1 \otimes Q$ and that
	$$\delta^1_{\tI{\C_1}' \otimes \tI{\C_2}} = \p_1 \otimes 1 + 1 \otimes \p_2 + \omega_1 \otimes Q + 1 \otimes Q \omega_2.$$
	Both $\p_1 \otimes 1$ and $\p_2 \otimes 1$ preserve the $Q$-grading and drop the $U$-grading by one. In a standard complex, $\gr(\omega) = (0,1)$ and $\gr(Q) = (0,-2)$ so both $\omega_1 \otimes Q$ and $1 \otimes Q \omega_2$ preserve the $U$-grading and drop the $Q$-grading by one.
\end{proof}

\begin{cor}
	Given two standard complexes $\C_1$ and $\C_2$, $\tI{\C_1}' \otimes \tI{\C_2}$ is a bigraded right $\mc{R}$ module such that 
	$\gr U = (-2,0)$, $\gr Q = (0,-2)$ and
	the differential
	$$\delta^1_{\tI{\C_1}' \otimes \tI{\C_2}} = (\p_1 \otimes 1 + 1 \otimes \p_2) + (\omega_1 \otimes 1 + 1 \otimes \omega_2) \cdot Q$$
	has homogeneous degree $-1$.
\end{cor}
\begin{proof}
	This follows immediately from the fact that $\sum x_i \otimes p_{ij} y_j = \sum x_i \otimes y_j \cdot p_{ij}$. 
\end{proof}

The product bigrading on $\tI{\C_1}' \otimes \tI{\C_2}$ can be extended to the precurve $\mc{F}(\tI{\C_1}' \otimes \tI{\C_2})$ by letting the generators $x^1$ on $s_1(a)$ and $x^2$ on $s_2(a)$ corresponding to $x$ have the same bigrading as $x$.

\begin{prop}\label{prop:bigrading_preserving_right_r_module_iso}
	For any standard complexes $\C_1$ and $\C_2$ there is a bigrading preserving chain isomorphism of right $\mc{R}$ modules
	between $ \tI{\C_1}' \otimes \tI{\C_2}$ and $\mc{G}(\gamma(\tI{\C_1}' \otimes \tI{\C_2})).$
\end{prop}

To prove this claim, the following definitions are needed:
\begin{define}
	Given a graphical representation of a precurve $(C, \{P_a\}, d)$, $x \in C.\iota_{s_1(a)}$ and $y \in C.\iota_{s_2(a)}$ are \textit{strand connected} if there is a strand in $N(a)$ running between $x$ and $y$.
\end{define}
For the definition of a strand, see Definition~\ref{def:strand}.
\begin{define}
	Given a graphical representation of a precurve $(C, \{P_a\}, d)$, two elements $x,y \in C$ on an arc $a$ are \textit{immediately connected} if there is a crossover arrow in the decomposition connecting the strands on which and $x$ and $y$ lie.
\end{define}
\begin{define} \label{def:connected}
	Given a graphical representation of a precurve $(C, \{P_a\}, d)$, two elements $x, y \in C$ on an arc $a$ are \textit{connected}  if there is a sequence of generators $\{ x= z_1, \cdots, z_n = y\}$ such that $z_i$ and $z_{i+1}$ are immediately connected.
\end{define}

Now that everything is finally in place, it is possible to prove Proposition~\ref{prop:bigrading_preserving_right_r_module_iso}.

\begin{proof}[Proof of Lemma~\ref{prop:bigrading_preserving_right_r_module_iso}]
	First it will be shown that the chain isomorphisms S1 and S2 which take $\mc{F}(\tI{\C_1}' \otimes \tI{\C_2})$ to a simply-faced precurve preserve the bigrading. Then it will be shown that the moves T1, T2, T3, M1, M2, and M3 from the arrow-sliding algorithm preserve the bigrading. Once this is shown, then it will follow that the composite chain isomorphism is a bigrading preserving chain isomorphism of right $\mc{R}$ modules.

	To show that the reduction of $\mc{F}(\tI{\C_1}' \otimes \tI{\C_2})$ to a simply-faced precurve is bigrading preserving, it will be shown that as S1 and S2 preserve the following three properties:
	\begin{enumerate}[(I)]
		\item 
			two generators in $N(a)$ are connected (in the sense of Definition~\ref{def:connected}) only if they have the same bigrading,
		\item
			the differential has homogeneous degree $-1$, and
		\item
			there is a component 
			$x.\iota_{s_1(a)} \overset{U^n}{\longrightarrow} y.\iota_{s_1(a)}$
			only if $x$ and $y$ have the same $Q$-grading, and there is a component 
			$x.\iota_{s_2(a)} \overset{Q}{\longrightarrow} y.\iota_{s_2(a)}$
			only if $x$ and $y$ have the same $U$-grading.
	\end{enumerate}
	Condition (I) is true in the case of $\mc{F}(\tI{\C_1}' \otimes \tI{\C_2})$ as no strands are connected yet. (II) and (III) are true in $\mc{F}(\tI{\C_1}' \otimes \tI{\C_2})$ by construction, and (III) is only when $\tI{\C_1}' \otimes \tI{\C_2}$ is considered as a right $\mc{R}$ module.
	
	Suppose that (I)-(III) hold true at a given step of the simplification of $\mc{F}(\tI{\C_1}' \otimes \tI{\C_2})$ to a simply-faced precurve, and S1a is about to be applied. It must be the case that S1a is being applied on $s_1(a)$, since all $Q$ arrows are of length one. Properties (II) and (III) guarantee that the chain isomorphism $(1+h)$ from Lemma~\ref{lem:s1_s_2_iso} induced by S1a is bigrading preserving. Note that the resulting precurve  preserves properties (I)-(III).
	
	Similarly, suppose that (I)-(III) hold true at a given step of the simplification to a simply-faced precurve, and S2a is about to be applied. Properties (II) and (III) guarantee that the chain isomorphism $(1+h)$ from Lemma~\ref{lem:s1_s_2_iso} induced by S1b is bigrading preserving. Note that the new crossover arrow must connect generators of equal grading by properties (II) and (III), so (I) is preserved. It is easy to check that (II) and (III) are preserved as well. 
	
	The argument to show that S2a and S2b preserve both bigrading and properties (I)-(III) is entirely parallel. This implies that $\mc{F}(\tI{\C_1}' \otimes \tI{\C_2})$ is taken to a simply-faced precurve which satisfies properties (I)-(III) by a chain isomorphism which is bigrading preserving. Before the arrow sliding algorithm can be applied, the strands must be ordered by homotoping the simply-faced precurve and applying M1. The chain isomorphism which M1 induces is given in Lemma~\ref{lem:m1_iso}, and it is a permutation of two generators which preserves properties (I)-(III) and is bigrading preserving.
	
	Now it only remains to perform the arrow sliding algorithm by successive application of T1, T2, T3, M1, M2, and M3. By the previous paragraph, it can be assumed that the initial ordered simply-faced precurve satisfies properties (I)-(III) and the total chain isomorphism from $\mc{F}(\tI{\C_1}' \otimes \tI{\C_2})$ to that precurve is bigrading preserving.
	
	T1, and T3 clearly preserve (I), (II) and (III). T2 preserves (I)-(III) in the forward direction (the deletion of adjacent identical arrows), but it is not immediatley evident that T2 preserves (I) in the reverse direction. However, a careful examination of the arrow sliding algorithm in \cite{hanselman2016bordered} will show that T2 is only applied in reverse if the two strands in question are already joined by a crossover arrow. Furthermore, the chain isomorphisms given by T1, T2, and T3 are the identity, as discussed in Lemma~\ref{lem:t1_t2_t3_iso}. It is already known that M1 preserves both bigrading and properties (I)-(III). Thus it only remains to check M2 and M3. 
	
	First it will be established that M2 preserves bigrading and properties (I)-(III). By property (I), a crossover arrow connects two strands only if the generators on each end of the strand have the same bigrading. Thus the generators connected by the crossover arrow on the LHS of M2 in figure~\ref{fig:precurve_simplification_p_a} have the same bigrading. Properties (II) and (III) guarantee that the generators connected by the crossover arrow on the RHS of M2 have the same bigrading as well. Therefore the isomorphism $(1 + h_1 + h_2)$ from Lemma~\ref{lem:m2_iso} which is induced by M2 is bigrading preserving. Also, properties (I)-(III) are preserved by M2. 
	
	Now it will be established that M3 preserves bigrading and properties (I)-(III). By property (I), the generators on the LHS of M3a, M3b, and M3c have the same bigrading. Thus the isomorphism $(1+h)$ described in Lemma~\ref{lem:m3_iso} is bigrading preserving. It is straightforward to verify that properties (I)-(III) are preserved by M3a, M3b, and M3c. 
	
	Now it has been shown that the chain isomorphism taking $\mc{F}(\tI{\C_1}' \otimes \tI{\C_2})$ to $\gamma(\tI{\C_1}' \otimes \tI{\C_2})$ is bigrading preserving. It is also a morphism of right $\mc{R}$ modules. If the composite map is denoted $\Psi$, then $\mc{G}(\Psi)$ is the desired bigrading preserving chain isomorphism of right $\mc{R}$ modules.
\end{proof}

This last lemma will be used to identify the final generator in the proof of Theorems~\ref{thm:p_homo} and \ref{thm:p_omega_homo}.

\begin{lem} \label{lem:standard_complex_fu_fq_homology_generators}
	For any standard complex $\C$, $H_*(\tCI; \F[Q])$ has a single tower generated by $[x_f(\C)]$. 
\end{lem}
\begin{proof}
	In the case of standard complexes, $\delta^1 = \p + Q \omega$, so $\delta^1/(U) = Q \omega$ and $H_*(\tCI; \F[Q])$ has a single tower generated by $[x_f(\C)]$. 
	If, for example 
		$$\C = \C(-,-2,-,3, \cdots, -7,+,-1),$$ 
	then 
	\begin{center}
		\begin{tikzcd}[sep=small]
		\tCI = C(\tCI; \mc{R}) = \bullet_{x_i(\C)} \arrow[r,red] & \bullet \arrow[r,"2"] & \bullet \arrow[r,red] & \bullet & \arrow[l, "3"] \cdots \arrow[r, "7"] & \bullet & \bullet \arrow[l,red] \arrow[r, "1"] & \bullet_{x_f(\C)} \\
		C(\tCI; \F[Q]) = \bullet_{x_i(\C)} \arrow[r,red] & \bullet  & \bullet \arrow[r,red] & \bullet & \cdots& \bullet & \bullet \arrow[l,red]  & \bullet_{x_f(\C)}
		\end{tikzcd}
	\end{center}
	and clearly $Q^{-1} H_*(\tCI; \F[Q])$ is generated by $[x_f(\C)]$.
\end{proof}

\begin{prop} \label{prop:chain_iso_preserves_bigrading_of_final_generator}
	Given two almost $\iota$-complexes $\C_1$ and $\C_2$, let $\gamma(\tI{\C_1}' \otimes \tI{\C_2})$ be the multicurve obtained from reducing the bigraded right $\mc{R}$-module structure of $\tI{\C_1}' \otimes \tI{\C_2}$. Then 	
	$$\gr_{\circ}(x_f(\C_1) \otimes x_f(\C_2)) = \gr_{\circ}(x_f(\gamma_0(\tI{\C_1}' \otimes \tI{\C_2})))$$
	for $\circ \in \{U,Q\}$.
\end{prop}
\begin{proof}
	By Proposition~\ref{prop:bigrading_preserving_right_r_module_iso}, there is a bigrading preserving chain isomorphism of right $\mc{R}$ modules 
	$$\Phi : \tI{\C_1}' \otimes \tI{\C_2} \to \mc{G}(\gamma(\tI{\C_1}' \otimes \tI{\C_2})).$$
	Since $\Phi$ is a chain isomorphism, 
	$$\Phi_* : H_* (\tI{\C_1}' \otimes \tI{\C_2}; \mc{R}) \to H_*(\mc{G}(\gamma(\tI{\C_1}' \otimes \tI{\C_2})); \mc{R})$$
	is an isomorphism of right $\mc{R}$ modules. By Lemma~\ref{lem:r_chain_iso_descent}, 
	$$(\Phi_Q)_* : H_* (\tI{\C_1}' \otimes \tI{\C_2}; \F[Q]) \to H_*(\mc{G}(\gamma(\tI{\C_1}' \otimes \tI{\C_2})); \F[Q])$$
	is an isomorphism of right $\F[Q]$ modules. Note that as a right $\mc{R}$ module, the differential in $\tI{\C_1}' \otimes \tI{\C_2}$ obeys the Leibniz rule, as
	$$\delta^1_{\tI{\C_1}' \otimes \tI{\C_2}} = \p_1 \otimes 1 + 1 \otimes \p_2 + (\omega_1 \otimes 1 + 1 \otimes \omega_2) \cdot Q.$$
	Quotienting by $U$ and applying the K\"unneth formula gives that as right $\mc{R}$ modules $H_*(\tI{\C_1}' \otimes \tI{\C_2} ; \F[Q]) \cong H_*(\tI{\C_1}' ; \F[Q]) \otimes H_*(\tI{\C_2} ; \F[Q])$. By Lemma~\ref{lem:standard_complex_fu_fq_homology_generators}, it follows that 
	the only tower in $H_* (\tI{\C_1}' \otimes \tI{\C_2}; \F[Q]) $ is generated by $[x_f(\C_1) \otimes x_f(\C_2)]$. Therefore the only tower in $H_*(\mc{G}(\gamma(\tI{\C_1}' \otimes \tI{\C_2})); \F[Q])$ is generated by $[x_f(\gamma_0(\tI{\C_1}' \otimes \tI{\C_2}))]$. It then must be that
	$$(\Phi_Q)_*[x_f(\C_1) \otimes x_f(\C_2)] = [x_f(\gamma_0(\tI{\C_1}' \otimes \tI{\C_2}))],$$
	which proves that 
	$$\gr_{\circ}(x_f(\C_1) \otimes x_f(\C_2)) = \gr_{\circ}(x_f(\gamma_0(\tI{\C_1}' \otimes \tI{\C_2})))$$
	for $\circ \in \{U,Q\}$.
\end{proof}

\begin{prop} \label{prop:left_and_right_actions_have_same_primitive_rep}
	For any standard complexes $\C_1$ and $\C_2$, let $\gamma^{\text{left}}(\tI{\C_1}' \otimes \tI{\C_2})$ and $\gamma^{\text{right}}(\tI{\C_1}' \otimes \tI{\C_2})$ be the multicurves obtained from the left and right $\mc{R}$ modules structures of $\tI{\C_1}' \otimes \tI{\C_2}$. Then
	$$\gamma^{\text{left}}_0(\tI{\C_1}' \otimes \tI{\C_2}) \cong \gamma_0^{\text{right}}(\tI{\C_1}' \otimes \tI{\C_2}).$$
\end{prop}
\begin{proof} 
	In the the proof of Proposition~\ref{prop:bigrading_preserving_right_r_module_iso}, the precurve corresponding to the right $\mc{R}$-module structure of $\tI{\C_1}' \otimes \tI{\C_2}$ was reduced to a multicurve $\gamma^{\text{right}}(\tI{\C_1}' \otimes \tI{\C_2})$ via the arrow sliding algorithm. Denote the chain isomorphism (and \textit{right} $\mc{R}$ module morphism) from $\tI{\C_1}' \otimes \tI{\C_2}$ to $\mc{G}(\gamma^{\text{right}}(\tI{\C_1}' \otimes \tI{\C_2}))$ by $\Phi$. 
	Recall that
	$$\omega_{\tI{\C_1}' \otimes \tI{\C_2}} \cdot Q = (\omega_1 \otimes 1 + 1 \otimes \omega_2) \cdot Q$$
	and
	$$Q \cdot \omega_{\tI{\C_1}' \otimes \tI{\C_2}} = Q \cdot (\omega_1 \otimes 1 + 1 \otimes \omega_2 + \omega_1 \otimes \omega_2).$$
 	Converting $\mc{G}(\gamma^{\text{right}}(\tI{\C_1}' \otimes \tI{\C_2}))$ to a left $\mc{R}$-module would involve including the contribution of $Q \omega_1 \otimes \omega_2$ in the new basis. As a precurve, this would add several new arcs around the $Q$-puncture to $\gamma^{\text{right}}(\tI{\C_1}' \otimes \tI{\C_2})$. Denote the new precurve obtained after $Q \omega_1 \otimes \omega_2$ is included by $M^{\text{left}}$. Note that $\Phi$ extends linearly to give a chain isomorphism of \textit{left} $\mc{R}$ modules $\Phi': \tI{\C_1}' \otimes \tI{\C_2} \to \mc{G}(M^{\text{left}})$.
 	
	Each of the added arcs in $M^{\text{left}}$ are length one wraps around the $Q$-puncture, and can be pushed into $N(a)$ via the moves S1b and S2b. These new arrows can be slid off of the primitive curve $\gamma^{\text{right}}_0(\tI{\C_1}' \otimes \tI{\C_2})$ by bringing each new arrow closer to a puncture until its endpoints diverge using T3, M2, and M3. Since $\gamma^{\text{right}}_0$ is a single curve connecting the $U$ and $Q$ punctures, this process will terminate. Note also that this sliding process could very easily add decoration to the other curves which are not primitive, but this is of no consequence up to local equivalence. The resulting primitive representative will be $\gamma^{\text{left}}_0(\tI{\C_1}' \otimes \tI{\C_2})$ since the chain isomorphism type of the primitive representative is unique up to local equivalence. Thus 
	$$\gamma^{\text{left}}_0(\tI{\C_1}' \otimes \tI{\C_2}) \cong \gamma_0^{\text{right}}(\tI{\C_1}' \otimes \tI{\C_2}).$$
\end{proof}

\begin{rem}
	Note also that the resulting chain isomorphism obtained by pushing off all of the new arrows with endpoints on the primitive representative may no longer be bigrading preserving, but this is also of no consequence because it does not change the isomorphism type of the primitive representative. Thus it may not be true that $\gamma^{\text{left}}(\tI{\C_1}' \otimes \tI{\C_2}) \cong \gamma^{\text{right}}(\tI{\C_1} \otimes \tI{\C_2})$.
\end{rem}

It can now be proven that $P$ and $P_\omega$ are homomorphisms.

\begin{proof}[Proof of Theorems \ref{thm:p_homo} and \ref{thm:p_omega_homo}]
	Again, let $\gamma_0^{\text{left}}(\tI{\C_1}' \otimes \tI{\C_2})$ and $\gamma_0^{\text{right}}(\tI{\C_1}' \otimes \tI{\C_2})$ obtained from the left and right $\mc{R}$ modules structures of $\tI{\C_1}' \otimes \tI{\C_2}$. By Proposition~\ref{prop:chain_iso_preserves_bigrading_of_final_generator}, it follows that
	$$\gr_{U}(x_f(\gamma_0^{\text{right}}(\tI{\C_1}' \otimes \tI{\C_2}))) = \gr_{U} (x_{f}(\C)) + \gr_U ( x_f(\C))$$
	and
	$$\gr_{Q}(x_f(\gamma_0^{\text{right}}(\tI{\C_1}' \otimes \tI{\C_2}))) = \gr_{Q} (x_{f}(\C)) + \gr_Q (x_f(\C)).$$
	Since $\gamma^{\text{left}}_0(\tI{\C_1}' \otimes \tI{\C_2}) \cong \gamma_0^{\text{right}}(\tI{\C_1}' \otimes \tI{\C_2})$ by Proposition~\ref{prop:left_and_right_actions_have_same_primitive_rep}, it then follows that 
	$$P(\C_1 \otimes \C_2) = P(\C_1) + P(\C_2) \qquad \text{ and } \qquad P_{\omega}(\C_1 \otimes \C_2) = P_{\omega}(\C_1) + P_\omega(\C_2)$$
	as the primitive representative of $\C_1 * \C_2$ is given by $\gamma_0^{\text{left}}(\tI{\C_1}' \otimes \tI{\C_2}).$
\end{proof}

\subsection{Shift maps and $\phi_n$ with precurves} \label{subsection:shift_maps_and_phi_n_with_precurves}
Now it will be reproven that each $\phi_n$ is a homomorphism. All that is left to do is to show that the shift maps from \cite{dai2018infiniterank} are endomorphisms of $\hI$. 
An equivalent definition of the shift maps $\sh_n$ can be given as follows:

\begin{define}
	Given any reduced precurve $M$ on the twice punctured disk, $\sh_n(M)$ replaces every arc around the $U$ puncture of length $m \geq n$ with an arc of length $m+1$.
\end{define}

Note that the $Q^2$-reduction of $\sh_n(\mc{F}(\tCI))$ is $\sh_n(\C)$ as defined in \cite{dai2018infiniterank} for any reduced almost $\iota$-complex $\C$.

\begin{thm}[\cite{dai2018infiniterank}] \label{thm:sh_n_endo}
	$\sh_n$ is an endomorphism of $\hI$ for $n \geq 1$. 
\end{thm}
\begin{proof}
	Consider the sequence of moves S1, S2, T1, T2, T3, M1, M2, and M3 taking $\mc{F}(\tI{(\C_1*\C_2)})$ to $\gamma(\tI{(\C_1 * \C_2)})$. 
	Since the relative lengths of the generators and arcs do not change, the exact same moves clearly take $\mc{F}(\tI{(\sh_n(\C_1) * \sh_n(\C_2))})$ to $\sh_n(\gamma((\C_1 * \C_2)\mc{I}))$. Thus $\gamma_0(\tI{(\sh_n(\C_1) * \sh_n(\C_2))}) = \sh_n(\gamma_0((\C_1 * \C_2)\tilde{\mc{I}}))$. This proves that $\sh_n(\C_1) * \sh_n(\C_2)$ is locally equivalent to $\sh_n(\C_1 * \C_2)$, which shows that $\sh_n$ is a homomorphism.
\end{proof}

\begin{proof}[Proof of Theorem \ref{thm:phi_n_homo}]
	Using Theorems \ref{thm:p_homo} and \ref{thm:sh_n_endo}, the rest is exactly as in \cite{dai2018infiniterank}. 
	By counting the length of the differentials, it can be shown that 
	$$P(\C) = \sum_{i=1}^{\infty} (-2i+1) \phi_i(\C),$$
	and
	$$P(\sh_n(\C)) = P(\C) - 2 \sum_{i=n}^{\infty} \phi_i(\C),$$
	from which it can quickly be shown with the inductive argument given in \cite{dai2018infiniterank} that $\phi_n : \hI \to \Z$ is a homomorphism for all $n$. 
\end{proof}

\begin{proof}[Proof of Theorem \ref{thm:p_omega_independent}]
		Suppose $P_\omega$ were in the span of $\{ \phi_n \}$. Then $P_\omega = \sum_{n=1}^{\infty} a_n \phi_n$. Since $P_\omega(\C(+,k)) = 1$ for all positive $k$, and $\phi_n(\C(+,k))$ is nonzero if and only if $n=k$, it follows $a_n = 1$ for all $n$. If this is the case, then $P_\omega(\C(-,k))= \sum_{n=1}^{\infty} \phi_n(\C(-,k)) = 1$, which is a contradiction since $P_\omega(\C(-,k)) = -1$. Surjectivity follows from the fact that any number of $a_i = +$ or $a_i = -$ can appear in a standard complex.
\end{proof}

\bibliographystyle{alpha}
\bibliography{biblio_no_links}

\end{document}